\documentclass[11pt, reqno, a4paper]{amsart}
\usepackage[normalem]{ulem}
\usepackage[utf8]{inputenc}
\usepackage{a4wide}
\usepackage[foot]{amsaddr}

\usepackage{algorithm}
\usepackage{algpseudocode}
\usepackage{todonotes}

\usepackage{dsfont}
\usepackage{upgreek}
\usepackage{amssymb, amsthm, amsmath, enumerate, physics, mathrsfs, amsfonts, esint, empheq}
\usepackage[colorlinks=true, linkcolor=black, citecolor=black, urlcolor=black]{hyperref}
 \usepackage{accents}
\usepackage{fullpage}

\usepackage{natbib}
\usepackage{siunitx,booktabs}
\usepackage{placeins}

\usepackage{tikz, graphicx, subcaption, pgfplots, xcolor}
\usepackage{rotating} 
\usepackage{multirow} 
\pgfplotsset{compat=1.18}
\usetikzlibrary{decorations.pathreplacing,calc,fit,intersections}

\allowdisplaybreaks
\numberwithin{equation}{section}

\newcommand{\R}{\mathbb{R}} 

\renewcommand{\H}{\mathcal{H}} 

\newcommand{\C}{\mathcal{C}} 
\renewcommand{\P}{\mathscr{P}} 
\newcommand{\Pac}{\mathscr{P}_{\mathrm{ac}}} 

\newcommand{\G}{\mathcal{G}} 

\newcommand{\X}{\mathcal{X}} 
\newcommand{\Y}{\mathcal{Y}} 

\newcommand{\T}{\mathcal{T}_c} 
\newcommand{\f}{f_{\mathrm{cor}}} 
\NewDocumentCommand{\Lag}{ m O{\vb*{w}} O{\vb{z}} }{L^{#1}_c\qty(#2, #3)} 

\NewDocumentCommand{\eps}{ o }{\ensuremath{\varepsilon\IfValueT{#1}{_{#1}}}}

\theoremstyle{plain} 

\newtheorem{lemma}{Lemma}[section]

\newtheorem{rmk}{Remark}[section]
\theoremstyle{definition}
\newtheorem{defn}{Definition}[section]

\usepackage{todonotes}

\begin{document}

\title{A Semi-Discrete Optimal Transport Scheme for the Semi-Geostrophic Slice Compressible Model}
\author[1]{Th\'eo Lavier$^*$} \author[2]{Beatrice Pelloni$^\dagger$}

\address[*]{Laboratoire IMATH, Université de Toulon, Toulon, France}

\address[$\dagger$]{Maxwell Institute for Mathematical Sciences and Department of Mathematics, Heriot-Watt University, Edinburgh, UK}

\date{}

\begin{abstract}

We develop a semi-discrete optimal transport scheme for the compressible semi-geostrophic equations, a system that plays an important role in modelling large-scale atmospheric dynamics and frontogenesis. 
Unlike the incompressible case, the compressible equations involve variable density and internal energy, but can be recast into a variational framework that naturally couples the dynamics with an optimal transport formulation. 
This is done by a change to the so-called geostrophic coordinates, via a transformation inspired by the incompressible case. 
The discrete version of this variational formulation provides the basis for a  numerical particle scheme. 
The implementation of this scheme presents considerable challenges, due to a non-quadratic cost function and parabolic $c$-Laguerre cells. 
To address these challenges, we use $c$-exponential charts to construct $c$-Laguerre tessellations efficiently, in which the cells become convex polygons. The time-continuous particle system conserves mass and energy, and the fully discrete scheme conserves them up to the time-discretisation error.
We study the convergence of the scheme numerically, including a single-seed benchmark and error analysis.
This work provides a significant new generalisation of existing semi-discrete optimal transport techniques, offering a robust and structure-preserving tool for simulating realistic atmospheric flows.

\end{abstract}

\maketitle

\section{Introduction}

\subsection{Background and motivation}

The compressible semi-geostrophic (SG) equations constitute a fundamental framework for modelling large-scale atmospheric dynamics, particularly in applications involving  frontogenesis and the evolution of geophysical flows past the time of weather front formation. 
In contrast with their incompressible counterparts, the compressible SG equations include variable density and internal energy, capturing essential thermodynamic processes. 
This added complexity is crucial for realistic atmospheric simulations but poses significant challenges in both the theoretical analysis and the numerical modelling of the equations. 

Over the past decades, optimal transport techniques have emerged as powerful tools for constructing structure-preserving numerical schemes in fluid dynamics \cite{Benamou:1998,Gallouet:2018}. 
Earlier work by \cite{Egan:2022, Lavier:2024} demonstrated the efficacy of semi-discrete optimal transport methods in the context of the incompressible SG equations, using Laguerre tessellations to generate transport maps that conserve mass and energy. 
The theoretical foundations and convergence results from \cite{Kitagawa:2019,Merigot:2021}, have consolidated the viability of numerical semi-discrete optimal transport methods. 
There has also been significant progress made on using full discrete optimal transport methods to simulate semi-geostrophic systems \cite{Benamou:2024,Benamou:2025,Carlier:2024}.
The fully compressible SG system was first studied in \cite{Cullen:2003}, a work that clarified both the mathematical intricacies and the meteorological importance of this system.
Cotter and Holm \cite{Cotter:2013} introduced a variational framework for vertical slice models, including a compressible Euler slice model, and Yamazaki and Cotter \cite{Cotter:2025} modified this model to obtain a frontogenesis test case for compressible dynamical cores, stressing the need to account for compressibility in realistic atmospheric models.

A numerical particle scheme based on semi-discrete optimal transport techniques, first implemented for the incompressible SG system \cite{Egan:2022, Lavier:2024}, was recently extended to the compressible system in \cite{Bourne:2025}, laying the foundation for the developments presented here.
Specifically, the present paper builds on the results presented in \cite{Bourne:2025} by developing a novel semi-discrete optimal transport scheme specifically tailored for the compressible SG equations. 
This is done by transforming the governing equations to geostrophic coordinates, an approach inspired by the seminal work of Hoskins \cite{Hoskins:1975},  whereby the complex interplay between density, temperature, and velocity is captured within a variational framework. 
A central challenge in this formulation is that the optimal transport problem associated with the compressible dynamics is defined by a cost function that incorporates gravitational potential and compressibility effects but is not reducible to the classical quadratic cost appearing in the incompressible case. 
This cost function implies that the resulting discrete version of the problem involves optimal tessellation with cells bounded by segments of parabolic curves. 
The associated theoretical and computational challenges are new and specific to the compressible case. 
To address these issues, our scheme incorporates internal energy as a dynamic variable, ensuring consistency with the first law of thermodynamics while preserving the underlying geometric and energetic structures of the system.

The contributions presented in this paper are in three parts. 
In the first part, we present what is, to our knowledge, the first semi-discrete optimal transport scheme for the 2D compressible SG equations in any setting, thereby advancing significantly the possibility to simulate realistic atmospheric flows. 
In the second part, we develop efficient computational techniques for constructing tessellations associated with non-quadratic cost functions, a key innovation that ensures numerical stability and accuracy. 
In the third part, we validate the convergence and the energy-conserving properties of the method through a series of numerical experiments, indicating the potential for application in large-scale atmospheric simulations. 
We do not present rigorous results on the convergence properties of the scheme; see Remark~\ref{rmk:convergence}.

The semi-geostrophic compressible slice model studied here, Eq.~\eqref{eq:SGSCM}, is obtained by applying the semi-geostrophic approximation to the compressible slice model of \cite{Cotter:2013, Cotter:2025}.

\subsection{Outline of the Paper}

This paper is organized as follows. 
In Section 2, we introduce the SG compressible slice equations, a reduced 2D model for a compressible, stratified, and rotating atmosphere. 
These equations are derived directly from the compressible Euler equations by applying the hydrostatic, vertical slice, and semi-geostrophic approximations. We show how the coordinate transformation to geostrophic variables leads naturally to a variational formulation involving an optimal transport problem.
In Section 3, we describe the semi-discrete optimal transport techniques and ideas underlying our numerical algorithm, detailing the construction of $c$-Laguerre tessellations and the introduction of $c$-exponential charts.
Finally, in Section 4 we specify the initial conditions and parameter choices used in our simulations and we validate the approach through numerical experiments, including a single-seed benchmark and error analysis using the Wasserstein-2 metric, before discussing representative simulation results.

We also include an appendix containing several key lemmas. As a consequence of these lemmas, volume integrals are reduced to boundary integrals via the divergence theorem, thereby enabling efficient computation of the dual functional and its derivatives. 

\section{Slice Compressible Model}\label{sec:Model}

In this section we derive a system of equations designed to model key compressible atmospheric dynamics within a two-dimensional vertical slice, while respecting hydrostatic balance and incorporating a semi-geostrophic approximation to isolate the large-scale balanced flow relevant to frontogenesis. 
It is notable that this model is not obtained through a reduction of the three-dimensional fully compressible SG system of \cite{Cullen:2003}, instead it is derived directly from the compressible Euler equations, following \cite{Cotter:2013}. 
The introduction of a vertical slice geometry is necessary because a quasi-2D solution, as can be found in the incompressible case \cite{Egan:2022}, cannot be derived by simple reduction of the 3D compressible model. 
More details on this derivation can be found in \cite[Chapter~5]{Lavier:2025}.

\subsection{Governing Equations and Physical Assumptions}

We begin with the three-dimensional \emph{compressible Euler equations}, formulated for a stratified, rotating atmosphere. 
The state of the system at a point $\vb{x}=(x_1,x_2,x_3)$, corresponding to longitude, latitude, and altitude, respectively, is described by the velocity field $\vb{u} = (u, v, w)$, the density $\rho$, and the potential temperature $\Theta$.
The thermodynamic state is closed by the equation of state for a dry ideal gas, which relates the pressure $p$ to the density and temperature via the Exner pressure $\Pi$ and is given by
\begin{equation}\label{eq:stateeq}
    \begin{split}
        p&=\rho R_d \Theta\Pi, \\
        \Pi(\rho, \Theta)&=\qty(\frac{\rho R_d\Theta}{p_0})^{\gamma-1},
    \end{split}
\end{equation}
where $p_0$ is a reference pressure, $\gamma=\flatfrac{c_p}{c_v}$ is the adiabatic index, and $R_d=c_p-c_v$ is the specific gas constant for dry air.

We adopt from the start the hydrostatic approximation, which assumes a balance between the vertical pressure gradient and gravity. 
Under this assumption, the governing equations in a reference frame rotating with frequency $\f$ about the vertical axis $\vu{e}_3$ are the following:
\begin{equation}\label{eq:fullrotatingsystem}
    \begin{cases}
        D_t \vb{u} + \f \vu{e}_3\times\vb{u}+c_p\Theta\nabla\Pi+g\vu{e}_3 = 0, \\
        \partial_t \rho + \nabla\cdot\qty(\rho\vb{u}) = 0, \\
        D_t \Theta = 0,
    \end{cases}
\end{equation}
where the material derivative $D_t$ is given by :
\begin{equation}
    D_t = \frac{\partial}{\partial t} + \mathbf{u} \cdot \nabla = \frac{\partial}{\partial t} + u \frac{\partial}{\partial x_1} + v \frac{\partial}{\partial x_2} + w \frac{\partial}{\partial x_3}.
\end{equation}

\subsubsection{Vertical Slice and Semi-Geostrophic Approximation}

To derive the model analysed in this paper, we first apply the \emph{vertical slice approximation}. 
As in \cite{Cotter:2013}, we consider a two-dimensional domain
\begin{equation}
    \X:=[-L, L) \times \qty[0,H],
\end{equation}
which represents a vertical slice of the atmosphere in the $x_1$-$x_3$ plane. 
The model is periodic in the $x_1$-direction, for all fields, and we assume rigid boundaries at the top and bottom.
The velocity field $\vb{u}=(u, v, w)$ is decomposed into its in-slice component $\vb{u}_S=(u,w)$, and the transverse component, $v$.  

The key assumption made to derive the model is that all fields are independent of the transverse coordinate $x_2$, except for a background temperature gradient that is necessary to support baroclinic instability.
Accordingly, the in-slice velocity $\vb{u}_S$, density $\rho$, and the in-slice potential temperature $\theta$ are assumed to depend only on $(x_1,x_3,t)$.
Following \cite{Cotter:2013, Cotter:2025}, the total potential temperature $\Theta$ is decomposed as
\begin{equation}\label{eq:ThetaFull}
    \Theta(x_1,x_2,x_3,t)=s x_2+\theta(x_1,x_3,t),
\end{equation}
where $sx_2$ is a background temperature gradient in the transverse direction. Here $s=-\theta_0\f\Lambda/g$, where $\theta_0$ is a constant reference temperature and $\Lambda$ is the vertical shear of the associated thermal wind \cite{Cotter:2025}.
The in-slice potential temperature $\theta$ is the full potential temperature in the slice $x_2=0$, not a small perturbation. 
In particular, it is of the order of $\theta_0$ and positive at the initial time (see Section~\ref{sec:ICs}).
The model is then formulated in terms of $\theta$.
As in \cite{Cotter:2025}, the Exner pressure in the slice model is evaluated with the in-slice temperature, $\Pi=\Pi(\rho,\theta)$, using Eq.~\eqref{eq:stateeq} with $\Theta$ replaced by $\theta$. 
This is not physically exact, but it is necessary to obtain a closed vertical slice model with conserved energy and potential vorticity \cite{Cotter:2013, Cotter:2025}.

We then apply the hydrostatic and semi-geostrophic approximations. 
The vertical momentum equation is replaced by hydrostatic balance. 
The along-slice momentum equation is replaced by geostrophic balance of the transverse velocity $v$, while the material derivative of $v$ is retained in its own equation. 
This is the geostrophic momentum approximation of Hoskins \cite{Hoskins:1975}, which leads to the semi-geostrophic equations \cite{Cullen:2003}.
The momentum equations in the full model are thus replaced by the following balance relations:
\begin{align}
    v&=\frac{1}{\f}c_p\theta\partial_{x_1}\Pi, \label{eq:balancerelations1}\\
    \theta &=-\frac{g}{c_p\partial_{x_3}\Pi}.\label{eq:balancerelations2}
\end{align}

The compressible Euler slice model of \cite[Section~5]{Cotter:2013} is obtained from Eq.~\eqref{eq:fullrotatingsystem} under the slice assumptions above.
We follow Yamazaki and Cotter \cite{Cotter:2025} in adding a constant $\Pi_0$ to the Exner pressure in the transverse momentum and energy equations, to control the mean horizontal flow and make the model suitable as a frontogenesis test case. 
Applying the SG approximation (Eq.~\eqref{eq:balancerelations1} and Eq.~\eqref{eq:balancerelations2}) to that model gives the \emph{semi-geostrophic slice compressible model}. 
To our knowledge, this system was first written down in \cite[Chapter~5]{Lavier:2025} and in the present work.
The model is given by
\begingroup
\setlength{\jot}{2pt}
\begin{subequations}\label{eq:SGSCM}
\begin{empheq}[left=\empheqlbrace]{alignat=3}
    &D_tv+\f u=sc_p\qty(\Pi-\Pi_0)                                   
    &&\quad \mathrm{in}\;\X\times[0,t_f], \label{eq:2Dcompmom} \\
    &D_t\theta + sv =0                                                
    &&\quad \mathrm{in}\;\X\times[0,t_f], \label{eq:2Dcomptemp} \\
    &D_t\rho+\rho\nabla\cdot\vb{u}_S=0               
    &&\quad \mathrm{in}\;\X\times[0,t_f], \label{eq:2Dcompcont} \\
    &c_p\theta\nabla\Pi=\qty(\f v,-g)^T                                 
    &&\quad \mathrm{in}\;\X\times[0,t_f], \\
    &\vb{u}_S\cdot\vu{n} = 0         
    &&\quad \mathrm{on}\; \X\times[0,t_f] \cap \{x_3=0, H\}, \label{eq:rigid_lid_bc} \\
    &(\vb{u}_S, v, \rho, \theta)(-L,x_3)=(\vb{u}_S, v, \rho, \theta)(L,x_3)
    &&\quad \mathrm{on}\; \X\times[0,t_f] \label{eq:Periodic}.
\end{empheq}
\end{subequations}
\endgroup
Here $D_t=\frac{\partial}{\partial t}+\vb{u}_S\cdot\nabla_2=\partial_t + u\partial_{x_1}+w\partial_{x_3}$ is the in-slice material derivative and $\Pi_0$ is a constant reference value of the Exner pressure.
In \cite{Cotter:2013}, $\Pi_0=0$.
In \cite{Cotter:2025}, $\Pi_0=0.864$ was tuned by numerical experiments to minimise the horizontal drift of the front.
Here we take $\Pi_0$ to be the initial $\sigma$-weighted average of the Exner pressure,
\begin{equation}\label{eq:RefPress}
    \Pi_0 = \frac{\int_{\X} \Pi(\vb{x},0)\,\dd\sigma_0(\vb{x})}{\int_{\X}\dd\sigma_0}, \qquad \sigma_0=\rho(\cdot,0)\,\theta(\cdot,0).
\end{equation}
By Eq.~\eqref{eq:GeoVel} below, this choice makes the initial mass-weighted mean of the horizontal geostrophic velocity vanish.
For the initial data of Section~\ref{sec:ICs}, it gives $\Pi_0=0.8602$, close to the tuned value of \cite{Cotter:2025}, without any tuning.
The choice of $\Pi_0$ only fixes a horizontal reference frame: by Lemma~\ref{lem:Pi0} below, solutions corresponding to two different constants differ by a horizontal translation at constant speed.
In this formulation, the transverse velocity $v$ is influenced by both the Coriolis force and the background shear $s$.

The existence of a solution of this system is an open problem. 
In what follows, we assume this existence and focus on deriving and studying a particle scheme for its numerical approximation.

\subsection{Formulation in geostrophic variables.}\label{sec:geovars}

For the analysis of SG, both incompressible and compressible, the key step is the formulation in a different set of variables. 
These variables were first introduced by Hoskins in \cite{Hoskins:1975}, but the full potential of this formulation was only understood in groundbreaking work of Benamou and Brenier \cite{Benamou:1998}. 
In these new, geostrophic variables, the incompressible 3D system is formulated as a continuity equation for the potential vorticity coupled with an optimal transport problem, with respect to a quadratic cost, between a fixed source measure and the time-dependent measure defined by the potential vorticity, denoted below by $\alpha(\vb{x},t)$. 

We use the following notation. 
The physical domain is $\X$, and the geostrophic (target) space is $\Y:=\R\times(0,\infty)$, the set of points $\vb{y}=(y_1,y_3)$ with positive vertical geostrophic coordinate. 
We write $\P(\cdot)$ for probability measures and $\Pac(\cdot)$ for those absolutely continuous with respect to the Lebesgue measure. 
For a measurable map $T:\X\to\Y$ and $\sigma\in\P(\X)$, the pushforward $T_\#\sigma\in\P(\Y)$ is defined by $T_\#\sigma(B)=\sigma(T^{-1}(B))$ for all Borel sets $B\subset\Y$.

In general, the optimal transport involves pushing forward, at each fixed time $t$, a source measure, denoted by $\sigma_t$, to the target measure  $\alpha_t$ in a way that minimises the overall cost $\T(\sigma_t,\alpha_t)$ as defined by the associated cost function $c(\vb{x},\vb{y})$. 
The subscript $t$ denotes a fixed time, and the optimal cost is
\begin{equation}\label{OTgen}
    \T(\sigma_t,\alpha_t)=\min_{(T_t)_\#\sigma_t=\alpha_t}\int_{\X}c(\vb{x},T_t(\vb{x})) \,\dd \sigma_t(\vb{x}).
\end{equation}
We formulate Eq.~\eqref{OTgen} with transport maps rather than transport plans. 
The existence of the optimal plan is guaranteed by the fact that the source measure $\sigma_t$ is absolutely continuous and the cost $c$ defined in Eq.~\eqref{eq:compcost2D} below satisfies the so-called twist condition, and thus the optimal plan is induced by a map (see, e.g., \cite{Merigot:2021} and \cite{Bourne:2025}).

The compressible 3D system can also be formulated by an analogous change of variables as a continuity equation coupled with an optimal transport problem \cite{Cullen:2003}.
However, in this case the source measure is time-dependent, and the optimal transport between this source measure and the potential vorticity $\alpha$ is to be considered with respect to a non-quadratic cost. 
Both these elements are new and require significant modifications as well as novel approaches both for the theory and for the numerical simulations of the compressible system. 

\subsubsection{The Source Measure $\sigma_t$.}

Let $\X\subset \R^2$ denote the 2D physical domain of the fluid.
The formulation and analysis of the 3D compressible SG system sketched in the previous paragraph motivate the introduction of the quantity $\sigma_t(\vb{x})=\rho(\vb{x},t)\theta(\vb{x},t)$ defined on $\X$. 
This quantity plays the role of the source measure in the associated optimal transport problem.

The equation describing the dynamics of  $\sigma_t$ can be obtained
by combining Eq.~\eqref{eq:2Dcompcont} and Eq.~\eqref{eq:2Dcomptemp} as follows
\begin{align*}
    &\rho\qty(D_t\theta+sv)+\theta\qty(D_t\rho+\rho\nabla\cdot\vb{u}_S)=0 \\
    &\Longleftrightarrow \quad \rho\partial_t\theta+\theta\partial_t\rho=-\rho\vb{u}_S\cdot\nabla\theta-\theta\vb{u}_S\cdot\nabla\rho-\rho\theta\nabla\cdot\vb{u}_S-\rho s v \\
    &\Longleftrightarrow \quad\rho\partial_t\theta+\theta\partial_t\rho=-\nabla\cdot\qty(\rho\theta\vb{u}_S)-\rho s v 
    \end{align*}
    and is given by
    \begin{equation}
\partial_t\sigma_t=-\nabla\cdot\qty(\sigma_t\vb{u}_S)-\frac{\sigma_t}{\theta}sv. \label{eq:PDESigma}
\end{equation}
We use Eq.~\eqref{eq:PDESigma} to show that the total mass of $\sigma_t$ is conserved. 
Positivity of $\sigma_t$ is discussed in Remark~\ref{rmk:positivity} below.
\begin{lemma}[Conservation of the total mass of $\sigma_t$]\label{lem:mass}
    Assume at time $t=0$ that $\sigma_0$ is a probability measure.
    Then for all $t$,
    \[
        \int_{\X}\,\dd\sigma_t(\vb{x})=1.
    \]
\end{lemma}
\begin{proof}
     Our goal is to show that the time derivative of $\int_{\X}\,\dd\sigma_t(\vb{x})$ is zero.
     Note the following relationship
     \begin{equation}\label{eq:sigmarelationships}
         \sigma_t\partial_{x_1}\qty(\sigma_t^{\gamma-1})=\qty(\gamma-1)\sigma_t\sigma_t^{\gamma-2}\partial_{x_1}\sigma_t=\qty(\frac{\gamma-1}{\gamma})\partial_{x_1}\sigma_t^{\gamma}.
     \end{equation}
     Using the geostrophic balance relation Eq.~\eqref{eq:balancerelations1} and the equation of state (Eq.~\eqref{eq:stateeq}), the source term in Eq.~\eqref{eq:PDESigma}, can be rewritten as
     \begin{equation}
         \frac{\sigma_t}{\theta}sv\stackrel{\eqref{eq:balancerelations1}}{=}\frac{sc_p}{\f}\sigma_t\partial_{x_1}\Pi\stackrel{\eqref{eq:stateeq}}{=}\frac{sc_p}{\f}\qty(\frac{R_d}{p_0})^{\gamma-1}\sigma_t\partial_{x_1}\sigma_t^{\gamma-1}\stackrel{\eqref{eq:sigmarelationships}}{=}C\partial_{x_1}\sigma_t^{\gamma},
     \end{equation}
     where the constant $C$ is defined as 
     \begin{equation}
         C = \frac{sc_p}{\f}\qty(\frac{R_d}{p_0})^{\gamma-1}\qty(\frac{\gamma-1}{\gamma}).
     \end{equation}
     Thus, the evolution equation for $\sigma_t$, Eq.~\eqref{eq:PDESigma} can be written as
     \begin{equation}
         \partial_t\sigma_t + \nabla\cdot(\sigma_t\vb{u}_S) = - C\partial_{x_1}\sigma_t^\gamma.
     \end{equation}
    A direct computation, using the boundary conditions Eq.~\eqref{eq:rigid_lid_bc} and Eq.~\eqref{eq:Periodic}, then shows that 
    \[
        \dv{}{t}\int_{\X}\,\dd\sigma_t(\vb{x}) =0.
    \]
    Therefore the total mass of $\sigma_t$ is conserved.
\end{proof}

\begin{rmk}[Positivity of $\sigma_t$]\label{rmk:positivity}
    Lemma~\ref{lem:mass} only shows that $\int_\X\dd\sigma_t$ is conserved. For $\sigma_t=\rho\theta$ to be a probability measure, we also need $\rho\geq0$ and $\theta\geq0$.
    Positivity of $\rho$ is preserved by the continuity equation Eq.~\eqref{eq:2Dcompcont} for sufficiently regular velocity fields.
    We note that $\theta$ is the full in-slice potential temperature, of the order of $\theta_0=300$\,K, rather than a perturbation (see Eq.~\eqref{eq:ThetaFull}). 
    Moreover, by the hydrostatic balance Eq.~\eqref{eq:balancerelations2}, $\theta>0$ is equivalent to $\partial_{x_3}\Pi<0$, that is, to the pressure decreasing with height.
    Since $\theta$ is not conserved along trajectories (it satisfies $D_t\theta=-sv$), we are not able to prove that $\theta$ remains positive for all times for weak solutions. 
    This is part of the open well-posedness theory of Eq.~\eqref{eq:SGSCM}.
    At the discrete level, positivity holds by construction. 
    The discrete source density is $\sigma=(f^*)'(w^i-c(\cdot,\vb{z}^i))\geq0$ on each cell (see Section~\ref{sec:SDOT}), and the vertical geostrophic coordinates satisfy $z^i_3>0$, which corresponds to $\theta>0$ by Eq.~\eqref{eq:2Dcompcoord}.
\end{rmk}
In the simulations of Section~\ref{sec:results}, $z^i_3$ indeed stays far from zero: over the 25 days of the main run, the lowest parcel temperature $\min_i z^i_3(t)/\lambda$ falls from $296$\,K at $t=0$ to $250$\,K at day 12.5 and $226.7$\,K at day 21, because $D_t\theta=-sv$, but remains strictly positive.

\subsubsection{ The Target Measure $\alpha_t$.}

We now consider the transformation to geostrophic variables analogous to the one considered in \cite[Chapter~5]{Lavier:2025}. 
For a fixed time $t$, this is the transformation $T_t:\X\to \Y$ given by
\begin{equation}\label{eq:2Dcompcoord}
    T_t(\vb{x}) = \mqty(x_1 + \f^{-1} v(\vb{x},t) \\ \frac{g}{\theta_0\f^2}\theta(\vb{x}, t)).
\end{equation}
The potential vorticity $\alpha_t$ is the density of the probability measure given by $\alpha_t=(T_t)_{\#}\sigma_t$.
In order to derive the PDE for the measure $\alpha_t$, we first take the material derivative of Eq.~\eqref{eq:2Dcompcoord}:
\begin{align}
    D_tT_t = \mqty(D_tx_1 + \f^{-1}D_tv \\ \frac{g}{\f^2\theta_0}D_t\theta) = \mqty(u + sc_p\f^{-1}\qty(\Pi-\Pi_0) - u \\ -\frac{sg}{\f^2\theta_0}v) =  \frac{sg}{\f\theta_0}\mqty( \frac{c_p\theta_0}{g}\qty(\Pi-\Pi_0) \\ \qty(\vb{x}-T_t)\cdot\vu{e}_1).
\end{align}
Define on $\X$ the velocity $\vb{W}(\cdot, t)$ as
\begin{align}
    \vb{W}:=J\mqty( \qty(\mathrm{Id}-T_t)\cdot\vu{e}_1 \\ \frac{c_p\theta_0}{g}\qty(\Pi_0-\Pi)),
\end{align}
where
\[
    J=\frac{sg}{\f\theta_0}\mqty(0 & - 1 \\ 1 & 0).
\]
Then
\[
    \partial_t T_t =-\vb{u}_S\cdot\nabla T_t + \vb{W}.
\]
Using the coordinate transformation we can define a new quantity $\mathcal S$, whose value can be extracted from $\theta$ and $v$, as 
\begin{equation}
  \mathcal{S}:=\frac{sv}{\theta}=
    \frac{sg}{\f\theta_0}\frac{\qty(T_t-\mathrm{Id})\cdot\vu{e}_1}{T_t\cdot\vu{e}_3}.
\end{equation}

The weak form of the PDE modelling the dynamics of $\alpha_t$ can now be derived by an explicit computation and, for any test function $\varphi$, is given by 
\begin{equation}\label{eq:2Dcompweakform}
    \dv{}{t}\int_{\Y}\varphi(\vb{y})\,\dd\alpha_t(\vb{y})=\int_{\Y} \nabla\varphi(\vb{y})\cdot \vb{w}(\vb{y},t)\,\dd\alpha_t(\vb{y}) -\int_{\Y}\varphi(\vb{y})S(\vb{y},t)\,\dd\alpha_t(\vb{y}),\quad \varphi\in\C^{\infty}_c(\Y),
\end{equation}
where 
\begin{align}
    \vb{w}(\vb{y},t)&=\vb{W}(T_t^{-1}(\vb{y}),t)=J\mqty( \qty(T_t^{-1}(\vb{y})-\vb{y})\cdot\vu{e}_1 \\ \frac{c_p\theta_0}{g}\qty(\Pi_0-\Pi(T_t^{-1}(\vb{y}),t))) \label{eq:GeoVel} \\
    S(\vb{y},t)&=\mathcal{S}(T_t^{-1}(\vb{y}),t)=\frac{sg}{\f\theta_0}\frac{\qty(\vb{y}-T_t^{-1}(\vb{y}))\cdot\vu{e}_1}{\vb{y}\cdot\vu{e}_3}, \label{eq:GeoMass}
\end{align}
which are the geostrophic velocity and the source/sink in geostrophic coordinates. 
We have thus obtained (the weak form of) the forced continuity equation
\begin{equation}\label{eq:alphaPDEwrong}
    \partial_t\alpha_t +\nabla\cdot\qty(\alpha_t\vb{w}) = -\alpha_t S.
\end{equation}

Notice that the vertical component of the geostrophic velocity $\vb{w}(\vb{y},t)$, which we denote as $w_3=\vb{w}\cdot\vu{e}_3$, can be extracted directly from Eq.~\eqref{eq:GeoVel} as
\begin{equation}
    w_3= - \frac{sg}{\f\theta_0}(\vb{y}-T^{-1}_t(\vb{y}))\cdot\vu{e}_1.
\end{equation}
Comparing this expression with Eq.~\eqref{eq:GeoMass}, we find an exact algebraic relationship between the source term $S$ and the vertical component of the geostrophic velocity, 
\begin{equation}
    S(\vb{y},t) = -\frac{w_3}{y_3}.
\end{equation}
Substituting this relationship back into Eq.~\eqref{eq:alphaPDEwrong} yields
\begin{equation}
    \partial_t\alpha_t +\nabla\cdot\qty(\alpha_t\vb{w}) = \alpha_t \frac{w_3}{y_3}.
\end{equation}
By dividing the entire equation by $y_3$ and applying the product rule for the divergence operator, we reveal that the forced continuity equation can be rewritten as a homogeneous conservative continuity equation for a modified measure 
\begin{equation}\label{eq:alphaPDE}
    \partial_t\frac{\alpha_t}{y_3}+\nabla\cdot\qty(\frac{\alpha_t}{y_3}\vb{w}) = 0.
\end{equation}
This demonstrates that while the potential vorticity $\alpha_t$ formally contains a source term due to compressibility effects, its evolution is in fact strictly slaved to the geostrophic vertical coordinate. 

\subsubsection{The Energy Functional and the Optimal Transport Problem.}

The total geostrophic energy associated to the semi-geostrophic approximation Eq.~\eqref{eq:SGSCM}, namely the sum of kinetic, potential, and internal energy contributions, is given by
\begin{equation}\label{eq:2Dcompenergy1}
    E_g(t)=\int_{\X} \frac{\rho}{2}v^2+g\rho x_3 -c_p\rho\Pi_0\theta +c_v\rho\Pi\theta \,\dd  \vb{x}.
\end{equation}
Define the constant
\begin{equation}\label{eq:lambda}
    \lambda:=\frac{g}{\theta_0\f^2},
\end{equation}
so that, by Eq.~\eqref{eq:2Dcompcoord}, $T_t(\vb{x})_3=\lambda\theta(\vb{x},t)$ and hence $\rho=\lambda\sigma_t/T_t(\vb{x})_3$. 
Using also $v=\f\qty(T_t(\vb{x})_1-x_1)$, $\int_\X\dd\sigma_t=1$ (Lemma~\ref{lem:mass}) and $c_v\rho\Pi\theta=\kappa\sigma_t^\gamma$ with
\begin{equation}\label{eq:kappa}
    \kappa:=c_v\qty(\frac{R_d}{p_0})^{\gamma-1},
\end{equation}
the energy Eq.~\eqref{eq:2Dcompenergy1} can be rewritten as
\begin{equation}\label{eq:energy_rewritten}
    E_g(t)=\lambda\,F(T_t,\sigma_t)-c_p\Pi_0,
\end{equation}
where, for a map $S:\X\to\Y$ and a density $\sigma$ on $\X$,
\begin{equation}\label{eq:Fdef}
    F(S,\sigma):=\int_{\X}\qty(\frac{\f^2}{2S(\vb{x})_3}\qty(x_1-S(\vb{x})_1)^2+g\frac{x_3}{S(\vb{x})_3})\sigma(\vb{x}) \,\dd \vb{x}+\int_{\X} f(\sigma(\vb{x})) \,\dd \vb{x},
\end{equation}
and the rescaled internal energy density $f$ is defined by
\begin{equation}\label{eq:fdef}
    f(s)=\begin{cases}
        \frac{\kappa}{\lambda} s^{\gamma} & \text{if}\;s\geq0, \\
        \infty & \text{otherwise}.
    \end{cases}
\end{equation}
Since $\lambda$ is a positive constant and $c_p\Pi_0$ is a constant, minimising $E_g$ is equivalent to minimising $F$.

\emph{The variational problem.} The source measure $\sigma_t$ is not prescribed: only the potential vorticity $\alpha_t$ is transported by the dynamics, and both the map and the source measure are determined from $\alpha_t$ by energy minimisation. This is the compressible version of the Cullen convexity principle \cite{Cullen:2003, Bourne:2025}: solutions of Eq.~\eqref{eq:SGSCM} should satisfy, at each time $t$,
\begin{equation}\label{eq:fullmin}
    F(T_t,\sigma_t)=\min_{\substack{\sigma\in\Pac(\X)}}\;\min_{S_\#\sigma=\alpha_t}F(S,\sigma)
    =\min_{\sigma\in\Pac(\X)}\qty{\T(\sigma,\alpha_t)+\int_\X f(\sigma)\,\dd\vb{x}}=:\mathcal{E}(\alpha_t).
\end{equation}
Here the minimisation is over all mass-preserving rearrangements $S$ \emph{and} over the source measure $\sigma$. 
The marginal $\sigma$ is free, rather than fixed as in the incompressible case. 
The constraint $\sigma\geq0$ is enforced by $f=\infty$ on $(-\infty,0)$, and $\int\dd\sigma=\int\dd\alpha_t=1$ follows from the pushforward constraint.
The transport cost $\T$ is defined in Eq.~\eqref{OTgen}, with cost function
\begin{equation} \label{eq:compcost2D}
    c(\vb{x},\vb{y}) = \frac{\f^2}{2y_3}\qty(x_1-y_1)^2+g\frac{x_3}{y_3}.
\end{equation}
Adding a constant to $c$ changes neither the minimisers of Eq.~\eqref{eq:fullmin} nor the tessellations below, since it is absorbed in the dual weights.

\emph{Dual formulation.} Let $f^*(r):=\sup_{s\in\R}\qty(rs-f(s))$ be the Legendre transform of $f$. For Eq.~\eqref{eq:fdef},
\begin{equation}\label{eq:fstar}
    f^*(r)=(\gamma-1)\frac{\kappa}{\lambda}\qty(\frac{\lambda\,r_+}{\gamma\kappa})^{\frac{\gamma}{\gamma-1}},\qquad (f^*)'(r)=\qty(\frac{\lambda\,r_+}{\gamma\kappa})^{\frac{1}{\gamma-1}}\geq0,
\end{equation}
where $r_+=\max\{r,0\}$. By \cite[Section~3]{Bourne:2025}, $\mathcal{E}(\alpha)$ admits the dual formulation
\begin{equation}\label{eq:dualcont}
    \mathcal{E}(\alpha)=\sup_{\psi}\qty{\int_\Y\psi\,\dd\alpha-\int_\X f^*\qty(\psi^c(\vb{x}))\,\dd\vb{x}},\qquad \psi^c(\vb{x}):=\sup_{\vb{y}\in\Y}\qty(\psi(\vb{y})-c(\vb{x},\vb{y})),
\end{equation}
and the optimal source density is $\sigma=(f^*)'(\psi^c)$. 
For a discrete measure $\alpha^N$, $\psi$ reduces to a vector of weights and Eq.~\eqref{eq:dualcont} becomes the dual functional Eq.~\eqref{eq:Gfunc} of Section~\ref{sec:numerics}.

\subsection{Semi-discrete optimal transport.}\label{sec:SDOT}

We now introduce a discretisation of the PDE~\eqref{eq:alphaPDE} by using the template of semi-discrete optimal transport.
Namely, we consider a set of points $\vb{z}_t=(\vb{z}^1_t,\ldots,\vb{z}^N_t)\in\Y^N$, where $\Y=\R\times(0,\infty)$ is the geostrophic space introduced in Section~\ref{sec:geovars}, and define a particle approximation of the potential vorticity $\alpha_t$ for fixed $t$
\begin{equation}
    \alpha_t^N(\vb{z})=\sum_{i=1}^N m_t^i\delta_{\vb{z}_t^i}.
\end{equation}
As was done in \cite{Lavier:2025}, to which we refer for details, we now use ideas of semi-discrete optimal transport. 
Given a discrete target measure such as $\alpha_t^N$, the optimal transport map with respect to the cost $c(\vb{x},\vb{y})$ becomes an optimal tessellation of $\X$, called a $c$-Laguerre tessellation. 
Let the set of \emph{seed vectors} be given by
\begin{equation}
    D^N := \left\{ \vb{z}= \left(\vb{z}^1,\ldots,\vb{z}^N \right) \in \Y^{N} : z_3^i\neq z_3^j \text{ whenever } i\neq j \right\},
\end{equation}
and the set of admissible mass given by
\begin{equation}
    \Delta^N := \left\{ \vb{m} = (m^1, \ldots, m^N) \in (0,1]^N : \sum_{i=1}^N m^i = 1 \right\}.
\end{equation}
The barycenters of the cells in this tessellation are known as the centroids, $\vb{C}^i$.
\begin{defn}[$c$-Laguerre tessellation]
\label{def:cLagCells} 
Given $(\vb*{w},\vb{z}) \in \R^N\times D^N$, the \emph{$c$-Laguerre tessellation of $\X$ generated by $(\vb*{w},\vb{z})$} is the collection of \emph{Laguerre cells} $\{\Lag{i}\}_{i=1}^N$ defined by
\begin{align*}
    \Lag{i}:=\qty{\vb{x}\in \X:c(\vb{x},\vb{z}^i)-w^i\leq c(\vb{x},\vb{z}^j)-w^j\; \forall\,j\in\qty{1,\ldots,N}},
    \quad i \in \{1,\ldots,N\}.
\end{align*}
\end{defn}
For $(\vb{z},\vb{m})\in D^N\times\Delta^N$, we denote by $\vb*{w}_*(\vb{z},\vb{m})\in\R^N$ the optimal weights, i.e.\ the maximiser of the dual functional $\G(\cdot,\vb{z})$ defined in Eq.~\eqref{eq:Gfunc}. Equivalently, $\vb*{w}_*$ is the unique, up to the normalisation discussed in \cite{Bourne:2025}, vector of weights such that the cells $\Lag{i}[\vb*{w}_*]$ satisfy the mass constraint Eq.~\eqref{eq:mass_constraint_num}. 
The corresponding source density is $\sigma=(f^*)'(w_*^i-c(\cdot,\vb{z}^i))$ on $\Lag{i}[\vb*{w}_*]$.
\begin{defn}[Centroid map]\label{def:centroid}
    Given $\vb{m} \in \Delta^N$ and a source measure $\sigma\in\P_{\mathrm{ac}}(\X)$, define the {\em centroid map} $\vb{C}:D^N\times\Delta^N\to(\R^2)^N$ by $\vb{C}(\vb{z}) \coloneqq (\vb{C}^1(\vb{z}),\ldots,\vb{C}^N(\vb{z}))$, where
    \begin{equation} \label{eq:centroid0}
         \vb{C}^i(\vb{z}, \vb{m})\coloneqq\frac{1}{m^i}\int_{L_c^i(\vb*{w}_*(\vb{z},\vb{m}),\vb{z})}\vb{x} \, \dd \sigma(\vb{x}).
    \end{equation}
    For clarity the dependence on the mass will be suppressed. 
\end{defn}

This says that $\vb{C}^i(\vb{z})$ is the centroid, or barycentre, of the set of points $T_{\alpha^N}^{-1}(\{ \vb{z}^i \})$ transported to $\vb{z}^i$ in the optimal transport from $\sigma$ to $\alpha^N$ for the cost $c$.

\subsubsection*{PDE Discretisation}

Starting with the velocity term on the right-hand side of Eq.~\eqref{eq:2Dcompweakform} we can write
\begin{equation}\label{eq:compvel}
    \int_{\Y}\nabla\varphi(\vb{y})\cdot\vb{w}(\vb{y},t)\,\dd\alpha_t(\vb{y})=\sum_{i=1}^N m_t^i \nabla \varphi(\vb{z}_t^i)\cdot J\mqty((\vb{C}^i(\vb{z}_t)-\vb{z}^i_t)\cdot\vu{e}_1 \\ \frac{c_p\theta_0}{g}E_I^i(\vb{z}_t, \vb{m}_t)),
\end{equation}
where
\begin{equation}
    E_I^i(\vb{z}_t, \vb{m}_t):=\frac{1}{m_t^i}\int_{L^i_c(\vb*{w}_*(\vb{z}_t, \vb{m}_t),\vb{z}_t)}\Pi_0-\Pi(\vb{x})\,\dd\sigma_t(\vb{x}).
\end{equation}
The source/sink term gives
\begin{equation}\label{eq:sourcesink}
    \int_{\Y}\varphi(\vb{y})S(\vb{y},t) \,\dd\alpha_t(\vb{y})=\sum_{i=1}^N m_t^i \varphi(\vb{z}^i_t)\frac{\qty(\vb{z}^i_t-\vb{C}^i(\vb{z}_t))\cdot\vu{e}_1}{\vb{z}^i_t\cdot\vu{e}_3}.
\end{equation}
The left-hand side of Eq.~\eqref{eq:2Dcompweakform} yields
\begin{equation}\label{eq:lefthandside2D}
    \dv{}{t}\int_{\Y} \varphi(\vb{y})\,\dd\alpha_t(\vb{y})=\sum_{i=1}^N\qty[ \dot{m}_t^i\varphi(\vb{z}_t^i)+m_t^i\nabla\varphi(\vb{z}_t^i)\cdot\dot{\vb{z}}^i_t].
\end{equation}
Combining equations Eq.~\eqref{eq:compvel}, Eq.~\eqref{eq:sourcesink}, and Eq.~\eqref{eq:lefthandside2D} gives
\begin{equation}
    \begin{split}
    &\sum_{i=1}^N\qty[ \dot{m}_t^i\varphi(\vb{z}_t^i)+m_t^i\nabla\varphi(\vb{z}_t^i)\cdot\dot{\vb{z}}^i_t] \\
    &=\sum_{i=1}^N m_t^i \nabla \varphi(\vb{z}_t^i)\cdot J\mqty((\vb{C}^i(\vb{z}_t)-\vb{z}^i_t)\cdot\vu{e}_1 \\ \frac{c_p\theta_0}{g}E_I^i(\vb{z}_t, \vb{m}_t))-\sum_{i=1}^N m_t^i \varphi(\vb{z}^i_t)\frac{\qty(\vb{z}^i_t-\vb{C}^i(\vb{z}_t))\cdot\vu{e}_1}{\vb{z}^i_t\cdot\vu{e}_3}.
    \end{split}
\end{equation}
Fix $i$.
First choose $\varphi$ such that $\varphi(\vb{z}^i_t)=1$, $\varphi(\vb{z}^{j}_t)=0$ for all $j\neq i$ and $\varphi$ constant in a neighbourhood of $\vb{z}^k_t$ for all $k$.
This gives an ODE for $m^i$.
Then choosing $\varphi$ to be a constant in a neighbourhood of $\vb{z}^j_t$ for all $j\neq i$ implies that $\nabla\varphi=0$ in a neighbourhood of $\vb{z}^j_t$ and taking $\varphi$ such that $\nabla\varphi(\vb{y})=\vb{y}$ in a neighbourhood of $\vb{z}^i_t$ then gives an ODE for $\vb{z}^i$.
The evolution of the $i$-th particle, for $i\in\qty{1,\ldots,N}$, is governed by the coupled ordinary differential equations:
\begin{equation}\label{eq:particle_system}
    \begin{cases}
        \dot{\vb{z}}^i_t= \vb{w}^i(t), \\
        \dot{m}_t^i= -m_t^iS^i(t),
    \end{cases}
\end{equation}
where, consistently with Eq.~\eqref{eq:GeoVel},
\begin{align*}
    \vb{w}^i(t)&=J\mqty((\vb{C}^i(\vb{z}_t)-\vb{z}^i_t)\cdot\vu{e}_1 \\ \frac{c_p\theta_0}{g}E_I^i(\vb{z}_t, \vb{m}_t)), \\
    S^i(t) &= \frac{sg}{\f\theta_0}\frac{\qty(\vb{z}^i_t-\vb{C}^i(\vb{z}_t))\cdot\vu{e}_1}{\vb{z}^i_t\cdot\vu{e}_3}.
\end{align*}

However, as observed in the continuous formulation, this system contains an exact cancellation. 
From the definition of $J$, the vertical velocity of the $i$-th particle is explicitly
\begin{equation}
    \dot{z}^i_{t,3}= \frac{sg}{\f\theta_0}\qty(\vb{C}^i(\vb{z}_t)-\vb{z}^i_t)\cdot\vu{e}_1.
\end{equation}
Comparing this to the source term $S^i(t)$, we find the discrete analogue of the relationship derived in the continuous section
\begin{equation}
    S^i(t)=-\frac{\dot{z}^i_{t,3}}{z^i_{t,3}}.
\end{equation}
Substituting this directly into the mass evolution equation gives:
\begin{equation}
    \dot{m}^i_t=m^i_t\frac{\dot{z}^i_{t,3}}{z^i_{t,3}}\implies \frac{\dot{m}^i_t}{m^i_t}=\frac{\dot{z}^i_{t,3}}{z^i_{t,3}}.
\end{equation}
Integrating this expression reveals an exact algebraic relationship between the mass of the $i$-th cell and its vertical position:
\begin{equation}\label{eq:mass_height}
    m^i_t=m_0^i\frac{z^i_{t,3}}{z^i_{0,3}}.
\end{equation}
Therefore, the dynamic evolution of the system is completely determined by the ordinary differential equations for the particle positions $\vb{z}^i_t$, with the cell masses $m^i_t$ evolving proportionally to their vertical coordinates.
This fundamentally differs from the incompressible cases, where the masses $m_i$ are constant in time.

The total mass $\sum_i m^i_t$ is conserved by the time-continuous system Eq.~\eqref{eq:particle_system}. 
Indeed, by Eq.~\eqref{eq:mass_height} and the expression for $\dot z^i_{t,3}$,
\[
    \dv{}{t}\sum_{i=1}^N m^i_t=\frac{sg}{\f\theta_0}\sum_{i=1}^N\frac{m^i_t}{z^i_{t,3}}\qty(\vb{C}^i(\vb{z}_t)-\vb{z}^i_t)\cdot\vu{e}_1 ,
\]
and the right-hand side vanishes. 
By the optimality condition $\sigma=(f^*)'(w_*^i-c(\cdot,\vb{z}^i))$ on each cell, $(z^i_1-x_1)\f^2/z^i_3=\partial_{x_1}c(\vb{x},\vb{z}^i)$, so the sum equals $-\frac{1}{\f^2}\int_\X\sigma\,\partial_{x_1}f'(\sigma)\,\dd\vb{x}$ up to a constant factor. 
Since $\sigma$ is continuous across cell boundaries and periodic in $x_1$, this integral is zero, which is the discrete analogue of Lemma~\ref{lem:mass}. 
After time discretisation, this conservation holds only up to the time-discretisation error (see Section~\ref{sec:conservation}).
\begin{lemma}[Role of $\Pi_0$]\label{lem:Pi0}
    Let $(\vb{z}_t,\vb{m}_t)$ be a solution of Eq.~\eqref{eq:particle_system} with reference constant $\Pi_0$. Let $\Delta\in\R$ and $V:=-\frac{sc_p}{\f}\Delta$. Then $\widehat{\vb{z}}^i_t:=\vb{z}^i_t+Vt\,\vu{e}_1$, $\widehat{\vb{m}}_t:=\vb{m}_t$ is a solution of Eq.~\eqref{eq:particle_system} with reference constant $\Pi_0+\Delta$.
\end{lemma}
\begin{proof}
    The cost Eq.~\eqref{eq:compcost2D} depends on $x_1$ and $y_1$ only through $x_1-y_1$, and $\X$ is periodic in $x_1$. Translating all seeds by $Vt\,\vu{e}_1$, with the same weights and masses, therefore translates every Laguerre cell and the source density $\sigma$ by $Vt\,\vu{e}_1$. 
    Hence $\vb{C}^i$ is translated by the same vector, $(\vb{C}^i-\vb{z}^i)\cdot\vu{e}_1$ is unchanged, and so is $\dot z^i_3$. 
    Since $\Pi$ is a function of $\sigma$, its average over each cell is also unchanged, so replacing $\Pi_0$ by $\Pi_0+\Delta$ changes $E_I^i$ into $E_I^i+\Delta$. 
    By Eq.~\eqref{eq:particle_system} and the definition of $J$, $\dot z^i_1=-\frac{sc_p}{\f}E_I^i$, which therefore increases by $V$ for every $i$.
\end{proof}
The same argument applies to the continuous system Eq.~\eqref{eq:SGSCM}, where a change of $\Pi_0$ is compensated by a uniform horizontal velocity $u\mapsto u+V$ in Eq.~\eqref{eq:2Dcompmom}, i.e.\ a Galilean change of frame. 
In particular, the choice of $\Pi_0$ does not affect the shape of the solution, only its horizontal position, and it cannot shear or elongate the front.

\section{Numerical Method}\label{sec:numerics}
In this section we describe the numerical implementation  of this 2D slice model. The numerical simulation of the solutions of the ODE system Eq.~\eqref{eq:particle_system} constitutes the first step towards a full 3D discretisation scheme for the compressible semi-geostrophic system.  

\subsection{Optimal Transport Framework}

The discretisation derived in Section~\ref{sec:Model} reduces the evolution of the atmosphere to a system of ODEs, Eq.~\eqref{eq:particle_system} governing the particle positions $\vb{z}^i$, while the masses $m^i$ are updated exactly at each timestep through their algebraic dependence on the vertical coordinate. 
However, evaluating the driving vector fields, specifically the centroids $\vb{C}^i$ and internal energy terms $E_I$, requires the computation of  the exact $c$-Laguerre tessellation at every timestep. 

While particle positions $\vb{z}$ are known from the time-integration, the weights $\vb*{w}\in\R^N$ defining the cells are implicit. 
They must be uniquely determined to satisfy the mass consistency condition required by the particle approximation 
\begin{equation} \label{eq:mass_constraint_num} 
    \mathcal{M}^i(\vb*{w}) := \int_{L_c^i(\vb*{w}_*, \vb{z})} \,\dd\sigma = m^i, \quad \forall\,
    i=1, \dots, N. 
\end{equation}
Finding the weights that satisfy Eq.~\eqref{eq:mass_constraint_num} is equivalent to solving the dual formulation of the semi-discrete optimal transport problem. 
We recover the optimal weights $\vb*{w}_*$ by maximizing the Kantorovich dual function 
\begin{equation}\label{eq:Gfunc}
    \G(\vb*{w},\vb{z})=\sum_{i=1}^N \qty(w^im^i - \int_{L^i_c(\vb*{w},\vb{z})} f^*\qty(w^i-c(\vb{x},\vb{z}^i))\,\dd\vb{x}),
\end{equation}
where $f^*$ is the Legendre transform, given by Eq.~\eqref{eq:fstar}, of the internal energy density $f$. 
This is the discrete version of the dual problem Eq.~\eqref{eq:dualcont}: for $\alpha=\alpha^N$, the $c$-transform $\psi^c(\vb{x})=\max_i\qty(w^i-c(\vb{x},\vb{z}^i))$ equals $w^i-c(\vb{x},\vb{z}^i)$ on the Laguerre cell $\Lag{i}$. 
At the maximiser $\vb*{w}_*$, the optimal source density is $\sigma=(f^*)'(w_*^i-c(\cdot,\vb{z}^i))$ on $\Lag{i}[\vb*{w}_*]$.
The functional $\G$ is concave and smooth, allowing us to solve for $\vb*{w}_*$ efficiently using a damped Newton algorithm. 
The gradient of this functional is exactly the mass mismatch error,
\begin{equation}
    \nabla_{\vb*{w}} \G = m^i -  \mathcal{M}^i,
\end{equation}
while its Hessian is determined by the geometric interactions between adjacent cells (see Appendix~\ref{appendix:Integral_Reductions} and \cite{Bourne:2025}).

In the compressible setting, this problem is computationally demanding because the source measure $\sigma$ is non-uniform and evolves in time, while in addition the cost function is non-quadratic. 
To compute the necessary integrals for the gradient and Hessian efficiently, we exploit the specific geometry of the cost function, as detailed in the following section. 

\subsection{Geometry of the Cost Function}\label{sec:Geom_Cost}

The efficient minimisation of the dual functional $\G$ requires computing integrals over the Laguerre cells. 
Due to the form of the cost function, the boundaries of these cells are segments of curved surfaces, making construction of the tessellation computationally expensive and integration over the cells expensive and geometrically complex. 

To simplify the problem, we map the computation into a coordinate system where the cell boundaries become linear.
For a fixed $\vb{y}\in\Y$, define the diffeomorphism
\begin{equation}
    \Phi=\Phi(\cdot,\vb{y}):\mathbb{R}^2 \to \mathbb{R}^2
\end{equation}
called the \emph{$c$-exponential map} (see \cite[Remark 4.4]{Kitagawa:2019}), by
\begin{equation}
    \Phi(\vb{x};\vb{y}) := \exp_{\vb{y}}^c(\vb{x}) = (-D_y c(\cdot,\vb{y}))^{-1}(\vb{x}), 
\end{equation}
which explicitly can be written as
\begin{equation}
    \Phi(\vb{x};\vb{y})=\mqty(
    \displaystyle
    \frac{y_3}{\f^2} x_1 + y_1
    \\
    \displaystyle
    \frac{y_3^2}{g} \left( x_3 - \frac{x_1^2}{2\f^2} \right)
    ).
\end{equation}

This change of variables simplifies the geometric structure of the problem significantly. 
Specifically, we show that the cost function satisfies the so-called {\em Loeper's condition} \cite{Kitagawa:2019}, which 
ensures that the cells are convex in the transformed coordinates. 

\begin{lemma}[Loeper's condition] \label{lem:2DLoeper}
    Given any $\vb{y},\vb{z} \in \Y$, define  $u: \mathbb{R}^2 \to \mathbb{R}$ by
    \[
        u(\vb{x}) = c(\Phi(\vb{x};\vb{y}),\vb{y}) - c(\Phi(\vb{x};\vb{y}),\vb{z}).
    \]
    Then $u$ is affine. 
    In particular, $u$ is convex, and hence the 2D compressible semi-geostrophic cost function $c$ satisfies Loeper's condition.
\end{lemma}
\begin{proof}
    First we compute that, for all $\vb{x} \in \mathbb{R}^2$,
    \begin{equation*}
        c(\Phi(\vb{x};\vb{y}),\vb{y})=\frac{1}{y_3}\left(\frac{\f^2}{2}(\Phi_1(\vb{x};\vb{y})-y_1)^2+g\Phi_2(\vb{x};\vb{y})\right)= y_3 x_3.
    \end{equation*}
    Similarly,
    \begin{align*}
        c(\Phi(\vb{x};\vb{y}),\vb{z}) 
        & = \frac{1}{z_3}\left(\frac{\f^2}{2}(\Phi_1(\vb{x};\vb{y})-z_1)^2+g\Phi_2(\vb{x};\vb{y})\right)
        \\
        & = 
        \frac{\f^2}{2z_3} \left(\frac{y_3}{\f^2} x_1 + y_1-z_1\right)^2+\frac{g}{z_3} \left( \frac{y_3^2}{g} \left( x_3 - \frac{x_1^2}{2\f^2} \right) \right)
        \\
        & = \frac{\f^2}{2 z_3}\left(y_1-z_1\right)^2
        + \frac{y_3^2}{z_3} x_3
        + \frac{y_3x_1}{z_3} \left(y_1-z_1\right).
    \end{align*}
    Therefore $u$ is affine, as required. 
\end{proof}

This result has a crucial practical implication. 
In the coordinate system defined by $\Phi$, the Laguerre cell is a convex polygon defined by the intersection of half planes (for the 3D analogue see \cite{Bourne:2025}). 
Indeed, fix a reference point $\vb{y}\in\Y$. 
By the computation in the proof of Lemma~\ref{lem:2DLoeper}, $\vb{x}\mapsto c(\Phi(\vb{x};\vb{y}),\vb{z})$ is affine for \emph{every} $\vb{z}\in\Y$, and not only for $\vb{z}=\vb{y}$. 
Hence, for every pair $i\neq j$, the set $\{\vb{x}: c(\Phi(\vb{x};\vb{y}),\vb{z}^i)-w^i\leq c(\Phi(\vb{x};\vb{y}),\vb{z}^j)-w^j\}$ is a half-plane. 
Every cell $\Phi^{-1}(\Lag{i})$ is therefore an intersection of half-planes, intersected with $\Phi^{-1}(\X)$, in the \emph{same} chart. 
This is why a single chart suffices, whereas for a general cost satisfying Loeper's condition each cell is only guaranteed to be convex in its own chart \cite{Kitagawa:2019}. 
The map $\Phi(\cdot;\vb{y})$ is a global diffeomorphism of $\R^2$, so no restriction of the chart is needed.
This allows us to employ efficient computational geometry algorithms, specifically half-plane clipping, to compute the volume and moments required by the Newton solver.

Note that, a priori,  the periodicity in the $x_1$-direction complicates the construction of the tessellation. 
However, in the $x_1$ direction the cost function is the standard quadratic cost.
Hence the results of \cite{Bourne:2023} can be applied without substantial modification, while taking into account, in determining the correct periodic tessellation, that the $c$-exponential chart is not a periodic transformation.
In practice, every seed $\vb{z}^i$ within a fixed distance of the periodic boundary is given copies at $\vb{z}^i\pm2L\vu{e}_1$ that share its weight $w^i$, and the Newton system of Section~\ref{sec:algorithm} is folded onto the $N$ weights of the original seeds. 
The copies are needed because a translation by $2L\vu{e}_1$ in $\X$ is not a translation in the chart. 
The image of one period of $\X$ in the chart is cut into $64$ vertical strips, and the parabolic bottom and top of each strip are replaced by their chords. 
This moves the boundary of the domain by at most $\f^2(2L/64)^2/(8g)\approx0.1$\,m, and its effect on the weights decreases as the square of the strip width (Section~\ref{sec:singleseed}).

In practice, computing the integrals for the dual functional and its derivatives relies on this transformation.
As detailed in Appendix~\ref{appendix:Integral_Reductions}, we map the computation into a single $c$-exponential chart defined around a reference point $y=[0,1]^T$.
Within this chart, the Laguerre cells form convex polygons, allowing us to use the divergence theorem to reduce the required 2D volume integrals to 1D boundary line integrals.
This avoids expensive quadrature over the cell in favour of line integrals yielding cheap and accurate gradients and Hessians used in our Newton solver.
The transformation renders the parabolic cell boundaries as straight-edged polygons, while it bends the bottom and top of the physical domain. 
Figure~\ref{fig:cartoon} illustrates this with a small number of seeds. In the physical slice, the boundary between two cells $i$ and $j$ is the set where $c(\vb{x},\vb{z}^i)-w^i=c(\vb{x},\vb{z}^j)-w^j$. 
By Eq.~\eqref{eq:compcost2D}, the coefficient of $x_1^2$ in this equation is $\frac{\f^2}{2}\big(\frac{1}{z^i_3}-\frac{1}{z^j_3}\big)$ and that of $x_3$ is $g\big(\frac{1}{z^i_3}-\frac{1}{z^j_3}\big)$. 
Hence every boundary with $z^i_3\neq z^j_3$ is an arc of a parabola
\begin{equation*}
    x_3=a_{ij}+b_{ij}\,x_1-\frac{\f^2}{2g}\,x_1^2,
\end{equation*}
and all these parabolas have the same curvature $\f^2/g$.
Note that a boundary with $z^i_3=z^j_3$ is a vertical segment. 
For the parameters of Table~\ref{tab:params}, $\f^2/g=10^{-9}$\,m$^{-1}$, so an edge of horizontal extent $\ell$ deviates from its chord by $\f^2\ell^2/(8g)$. 
This is about $2.5$\,cm for an edge of $14$\,km, the width of a cell at $N=10368$, and about $500$\,m across the whole domain. 
The cells in Section~\ref{sec:results} are drawn in the physical slice, where they look like polygons although their edges are curved.

\begin{figure}[!ht]
    \centering
    \includegraphics[width=\linewidth]{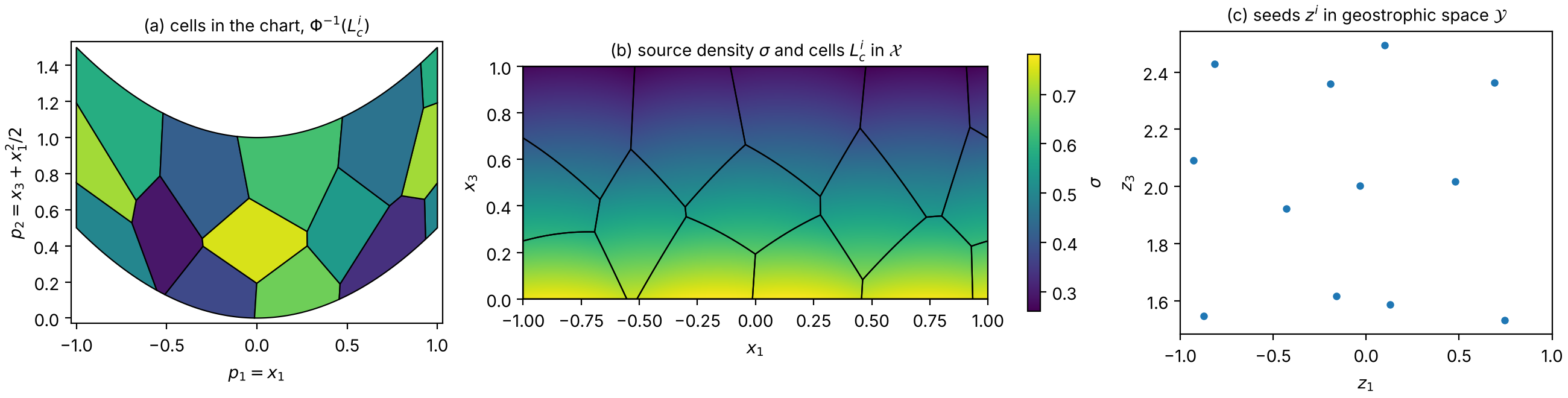}
    \caption{$c$-Laguerre cells for the cost Eq.~\eqref{eq:compcost2D} with the parameters Eq.~\eqref{eq:singleseedparams} ($\f=g=1$, $\gamma=2$, $\kappa=\frac12$), on $\X=(-1,1]\times[0,1]$, periodic in $x_1$, for $N=12$ seeds of equal mass $\frac{1}{12}$ and the optimal weights $\vb*{w}_*$. (a)~The cells in the $c$-exponential chart $\Phi^{-1}(\cdot\,;\vb{y})$, $\vb{y}=(0,1)$, where they are convex polygons and the image of $\X$ is bounded by two parabolas. (b)~The same cells in the physical slice, over the optimal source density $\sigma=(f^*)'(w_*^i-c(\cdot,\vb{z}^i))$ on $\Lag{i}$, which is continuous across the cell boundaries. Every boundary between two cells is an arc of a parabola $x_3=a+bx_1-\frac12x_1^2$, and cells that cross the periodic boundary $x_1=\pm1$ appear on both sides of the slice. (c)~The seeds $\vb{z}^i$ in geostrophic space.}
    \label{fig:cartoon}
\end{figure}
For more detailed discussions of $c$-exponential charts and their use for semi-discrete optimal transport methods for the semi-geostrophic equations see \cite{Bourne:2025, Lavier:2025}.

\subsection{Non-Dimensionalisation}\label{sec:nondim}

To non-dimensionalise the ODE in Eq.~\eqref{eq:particle_system}, we introduce  dimensionless variables, marked with tildes, using the following scalings :
\begin{align*}
    t &= \frac{\f\theta_0}{sg}\tilde{t}, \\
    x_1 &= L_0\widetilde{x}_1, \\
    x_2 &= H_1\widetilde{x}_2, \\
    z^i_1 &= L_0\widetilde{z}^i_1, \\
    z^i_2 &= H_2\widetilde{z}^i_2.
\end{align*}
Here, $L_0$ denotes the horizontal reference scale.
In Sections~\ref{sec:nondim}--\ref{sec:algorithm}, in the single-seed benchmark of Section~\ref{sec:singleseed} and in Appendix~\ref{appendix:Integral_Reductions}, we write the vertical coordinate as the second component of two-dimensional vectors, i.e.\ $\vb{x}=(x_1,x_2)$ stands for $(x_1,x_3)$ and $\vb{z}^i=(z^i_1,z^i_2)$ for $(z^i_1,z^i_3)$.
We utilize two distinct vertical reference scales $H_1$ for the physical height and $H_2$ for the geostrophic vertical coordinate, which accounts for the thermodynamic scaling inherent in the coordinate transformation.  
Substituting these relations into the cost function we obtain
\begin{align*}
    c\qty(\widetilde{\vb{x}},\widetilde{\vb{z}}^i) &= \frac{1}{H_2\widetilde{z}^i_2}\qty(\frac{\f^2}{2}L_0^2(\widetilde{x}_1-\widetilde{z}^i_1)^2+gH_1\widetilde{x}_2) \\
    &= \frac{\f^2L_0^2}{2H_2\widetilde{z}^i_2}\qty(\widetilde{x}_1-\widetilde{z}^i_1)^2+\frac{gH_1\widetilde{x}_2}{H_2\widetilde{z}^i_2}.
\end{align*}
Defining the non-dimensional parameters
\begin{align}
    F&=\sqrt{\frac{\f^2 L_0^2}{H_2}}, \\
    G&=\frac{gH_1}{H_2},
\end{align}
the cost function becomes
\begin{equation}
    c\qty(\widetilde{\vb{x}},\widetilde{\vb{z}}^i) = \frac{F^2}{2\widetilde{z}^i_2}\qty(\widetilde{x}_1-\widetilde{z}^i_1)^2+G\frac{\widetilde{x}_2}{\widetilde{z}^i_2}.
\end{equation}
In our simulations, we select the reference scales $L_0=10^7$\unit{m}, $H_1=10^5$\unit{m}, and $H_2=10^6$\unit{m}.
Together with $\f=10^{-4}$\,s$^{-1}$ and $g=10$\,m\,s$^{-2}$ from Table~\ref{tab:params}, these choices result in unity parameters $F=1$ and $G=1$, significantly simplifying the implementation.
Note that $F$ depends on $\f$, so a different Coriolis parameter would require a different choice of $L_0$ or $H_2$.
Note that these reference scales differ from the physical boundaries $L$ and $H$ defined in Table~\ref{tab:params}.

Ignoring the mass ODE as we showed that it is slaved to the particle position, the non-dimensional form of the particle evolution system Eq.~\eqref{eq:particle_system} is given by
\begin{equation}\label{eq:nondimpartsys}
    \begin{cases}
    \displaystyle
        \dot{\tilde{z}}_1^i=\frac{\eta_1}{m^i}\int_{\Lag{i}[\vb*{w}_*]} \qty(\Pi(\vb{x})-\Pi_0) \,\dd \sigma_t(\vb{x}) , \\
    \displaystyle
        \dot{\tilde{z}}^i_2=\eta_2\qty(\widetilde{C}^i_1-\widetilde{z}_1^i), 
    \end{cases}
\end{equation}
where $\widetilde{C}^i_1=C^i_1/L_0$ is the non-dimensional horizontal component of the centroid Eq.~\eqref{eq:centroid0}, and the dimensionless coupling constants $\eta_1$ and $\eta_2$ are given by :
\begin{equation}
    \eta_1=\frac{c_p\theta_0}{L_0g}\qq{and}\eta_2=\frac{L_0}{H_2}.
\end{equation}

\subsection{Description of the Numerical Algorithm}\label{sec:algorithm}

Our numerical scheme follows and extends the semi-discrete framework developed for incompressible flows outlined in \cite{Egan:2022, Lavier:2024} to the compressible setting. 
The algorithm evolves the discrete particle system by splitting the problem into two distinct phases at each timestep.
A geometric construction phase, i.e.~solving the optimal transport problem, and a dynamic evolution phase, i.e.~integrating the ODES.

The procedure for advancing the system from time $t_n$ to $t_{n+1}=t_n+\Delta t$ is detailed below. 
\vspace{0.1cm}
\paragraph{Step 1 : Optimal Transport} 
At time $t_n$, the particle positions $\vb{z}$ and the masses $m$ are fixed. 
To evaluate the driving vector fields, we must first construct the Laguerre tessellation that satisfies the mass constraint. 
Using a damped Newton solver we recover the unique optimal weights $\vb*{w}_*$ by maximising the dual functional $\G$ given in Eq.~\eqref{eq:Gfunc}. 
At $t=0$, the weights are initialised from the initial condition itself: by the optimality condition, $w^i-c(\vb{x},\vb{z}^i)=f'(\sigma(\vb{x}))=c_p\Pi(\vb{x})/\lambda$ on $\Lag{i}$, so we take $w^i=c(\vb{x}^i,\vb{z}^i_0)+c_p\langle\Pi\rangle^i/\lambda$, where $\vb{x}^i$ is the $\sigma_0$-barycentre of the region of the slice that defines $\vb{z}^i_0$ and $\langle\Pi\rangle^i$ the $\sigma_0$-weighted mean of the initial Exner pressure over that region (Section~\ref{sec:ICs}). 
The first Newton solve then makes the masses exact. 
We also use a rescaling algorithm suggested by \cite{Meyron:2019} to initialise the solver when no better guess is available. 
At later time steps, the solver starts from the optimal weights of the previous step, advanced to first order in the displacement of the seeds, $w^i\mapsto w^i+\nabla_{\vb{z}}c(\vb{C}^i,\vb{z}^i)\cdot\Delta\vb{z}^i$. 
This initial guess is not always admissible.
Near a front, two seeds can come so close that it leaves a cell with almost no mass, whereas the damped Newton method of \cite{Kitagawa:2019} needs every cell to have positive mass at the start. 
Whenever a cell receives less than $1\%$ of its target mass, the solver therefore also computes the rescaling guess of \cite{Meyron:2019}, adapted to the periodic copies of Section~\ref{sec:Geom_Cost}, and starts from it if its emptiest cell holds a larger fraction of its target mass. 
If a cell is still empty at the start, the first iterations are Levenberg--Marquardt steps. 
The Newton iterations stop when the relative mass error of every cell is below $10^{-8}$. 
At the front, round-off can stall them just above this level, so after $80$ iterations, or $15$ without a $1\%$ improvement, an error below $10^{-6}$ is accepted. 
As a safeguard, a step whose error stays above $10^{-6}$ would be solved by continuation, moving the seeds from their previous positions to the new ones in $2,4,\dots,64$ sub-steps; this was never needed in the simulations of Section~\ref{sec:results}. 
After each solve, the solver also checks that no periodic copy in the outer fifth of the band of copied seeds holds mass; otherwise it widens the band by a factor $1.5$ and solves the step again from the converged weights. 
This happened at most once per run, as the first front forms. 
In the main simulation ($N=10368$), a step takes on average $4$ to $7$ Newton iterations in the first days, $10$ to $16$ as the front forms and $30$ to $40$ after it (at most $79$; Section~\ref{sec:results}).

For each Newton iteration, the Laguerre cells are constructed in the $c$-exponential charts introduced in Section~\ref{sec:Geom_Cost}.
In these local coordinates, the cells are convex polygons, allowing for robust construction via half-plane intersection algorithms. 

\vspace{0.1cm}
\paragraph{Step 2 : Integral Evaluation}
A central challenge in this scheme is the accurate computation of the integrals defining the gradient and Hessian of $\G$, the centroids, and internal energy. 
Since the cells are polygonal in the $c$-exponential charts, we avoid expensive quadrature by converting area integrals into boundary line integrals where possible using the divergence theorem, see Appendix~\ref{appendix:Integral_Reductions}.

\vspace{0.1cm}
\paragraph{Step 3 : Dynamic Evolution}
Once the optimal weights $\vb*{w}_*$ and the corresponding cells are established, we evaluate the non-dimensional velocity fields from Eq.~\eqref{eq:nondimpartsys}.
We then advance the particle state using a standard explicit Adams-Bashforth scheme (AB2).
Crucially, the particle masses $m^i$ are not integrated numerically.
Instead, they are updated exactly at each timestep using the algebraic constraint
\begin{equation}
    m^i(t)=m^i(0)\frac{z^i_2(t)}{z^i_2(0)},
\end{equation}
derived in Section~\ref{sec:SDOT}.
At first glance, the presence of a source term in the governing PDE (Eq.~\eqref{eq:alphaPDEwrong}) suggests that a fully compressible scheme would require integrating a third coupled ODE for the mass evolution, which inherently risks the accumulation of numerical drift. 
By exploiting this exact analytical relationship, our algorithm completely bypasses the numerical integration of a mass equation. 
This avoids integrating an additional ODE for the masses and prevents accumulation of numerical errors associated with the mass evolution. 
Note however that, since the positions $z^i_2$ are integrated by AB2, the total mass $\sum_i m^i$ is conserved only up to the time-discretisation error. 
We report its evolution in Section~\ref{sec:conservation}.

This approach using the optimal transport solver preserves the global conservation properties of the time-continuous system, up to the time-discretisation error, while capturing the compressibility effects unique to this model.

\begin{rmk}[On convergence of the numerical scheme]\label{rmk:convergence}
    To our knowledge, no convergence result is available for a time-discretised semi-discrete scheme, even in the incompressible case. 
    For the three-dimensional incompressible SG equations, Bourne, Egan, Pelloni and Wilkinson \cite{Bourne:2022} showed that the semi-discrete particle system is well posed for well-prepared discrete initial data. 
    They also showed that a sequence of such discrete solutions, generated from quantised initial data, converges uniformly in time in the Wasserstein-2 distance to an energy-conserving weak solution, which yields the existence of weak solutions.
    The same strategy was carried out for the three-dimensional fully compressible SG system in \cite{Bourne:2025}, where the discrete solutions converge uniformly in time in the Wasserstein-1 distance.
    These are compactness results which concern the time-continuous particle system and a constructed approximating sequence. 
    Since uniqueness of weak solutions is open, they give neither convergence of arbitrary particle approximations nor convergence rates.
    For the slice model Eq.~\eqref{eq:SGSCM}, even the existence of weak solutions is open. 
    The model differs from the setting of \cite{Bourne:2025} in three ways: the domain is periodic in $x_1$, the equations contain the forcing terms involving $sv$ and $\Pi_0$, and the particle masses vary in time. 
    By Eq.~\eqref{eq:alphaPDE} and Eq.~\eqref{eq:mass_height}, the last difference reduces to the transport of the weighted measure $\alpha_t/y_3$ and an explicit algebraic update of the masses.
    While we do not attempt a convergence analysis here, Section~\ref{sec:convergence} provides numerical evidence of convergence with respect to both the timestep and the number of particles.
\end{rmk}

\section{Numerical Experiments}

In this section, we validate our numerical algorithm and demonstrate its capability to simulate compressible atmospheric dynamics. 
We begin by defining the experimental configuration, including the physical parameters and initial conditions used in our simulations.
We then present a benchmark study based on an analytical single-seed solution to quantify error convergence, followed by full-scale simulations of frontogenesis.

\subsection{Experimental Setup}

Our simulations follow the vertical slice frontogenesis test case adapted from \cite{Cotter:2025}.
The physical domain is a periodic channel in the horizontal direction $x_1$ with rigid boundaries at the top and bottom.

\subsubsection{Parameters}

We adopt the parameter choices listed in Table~\ref{tab:params}, which correspond to realistic atmospheric scales. 
They are those of the compressible test case of \cite[Section~2.2]{Cotter:2025}, which follow the incompressible Euler--Boussinesq Eady test case. 
The values of $c_p$ and $R_d$, which are not specified in \cite{Cotter:2025}, are the standard values for dry air.
In addition to these dimensional values, the initial condition is characterised by the non-dimensional Burger number $\mathrm{Bu}=0.5$ and the amplitude parameter $a=-7.5$, also taken from \cite{Cotter:2025}. 
These are not free parameters: $\mathrm{Bu}=\mathrm{Ro}/\mathrm{Fr}$, with Rossby number $\mathrm{Ro}=u_0/(\f L)=0.05$ and Froude number $\mathrm{Fr}=u_0/(N_{\mathrm{BV}}H)=0.1$ for the representative velocity $u_0=5$\,m\,s$^{-1}$, and $a$ sets the amplitude of the initial perturbation.
The transverse temperature gradient is $s=-\theta_0\f\Lambda/g$, where $\Lambda=10^{-3}$\,s$^{-1}$ is the vertical shear of the basic state in \cite{Cotter:2025}.
As in \cite{Cotter:2025}, the model is first integrated from the initial condition of Section~\ref{sec:ICs} until $\max_i|v^i|=3$\,m\,s$^{-1}$, which takes $48.5$\,h at $N=10368$ ($50$\,h in \cite{Cotter:2025}), and the clock is then reset to $t=0$.
Unless stated otherwise, times are measured from this reset, so they can be compared directly with \cite{Cotter:2025}. 
The resolution and timestep studies of Section~\ref{sec:convergence} start from the initial condition itself, without this breeding phase.

\begin{table}[!ht]
    \centering
    \begin{tabular}{|c||cl|}
    \hline
    Parameter & Value & Unit \\
    \hline
    $L$ & $10^6$ & \unit{m} \\
    $H$ & $10^4$ & \unit{m} \\
    $\f$ & $10^{-4}$ & \unit{s^{-1}} \\
    $N_{\mathrm{BV}}$ & $5\times10^{-3}$ & \unit{s^{-1}} \\
    $g$ & $10$ & \unit{m.s^{-2}} \\
    $p_0$ & $10^5$ & \unit{kg.m^{-1}.s^{-2}} \\
    $\theta_0$ & $300$ & \unit{K} \\
    $s$ & $-3\times10^{-6}$ & \unit{K.m^{-1}} \\
    $c_p$ & $1003.5$ & \unit{m^2.K^{-1}.s^{-2}} \\
    $R_d$ & $287.052874$ & \unit{m^2.K^{-1}.s^{-2}} \\
    \hline
    \end{tabular}
    \caption{Physical parameters used in the numerical simulations of compressiblefrontogenesis}.
    \label{tab:params}
\end{table}

\subsubsection{Initial Conditions}\label{sec:ICs}

The initial state is defined by a hydrostatic base state perturbed to induce frontogenesis. 
Recall from Eq.~\eqref{eq:ThetaFull} that the total potential temperature is $\Theta(\vb{x},t)=sx_2+\theta(x_1,x_3,t)$.
At the initial time $t=0$, the in-slice potential temperature $\theta$ is split into a hydrostatic base state $\overline{\theta}(x_3)$ and a perturbation $\widetilde{\theta}(x_1,x_3)$
\begin{equation}
    \theta(x_1,x_3,0) = \overline{\theta}(x_3) + \widetilde{\theta}(x_1,x_3).
\end{equation}
The base state is stratified according to the Brunt-V\"ais\"al\"a frequency $N_{\mathrm{BV}}$
\begin{equation}
    \overline{\theta}(x_3) = \theta_0 \exp\qty(\frac{N_{\mathrm{BV}}^2}{g}\qty(x_3-\frac{H}{2})).
\end{equation}
To trigger the dynamics, we impose the following perturbation
\begin{equation}
    \widetilde{\theta}(x_1, x_3) = \frac{\theta_0 a N_{\mathrm{BV}}}{g} \qty(-\qty(1-\frac{\mathrm{Bu}}{2}\coth{\frac{\mathrm{Bu}}{2}})\sinh{Z(x_3)}\cos\frac{\pi x_1}{L}-n \mathrm{Bu}\cosh{Z(x_3)}\sin\frac{\pi x_1}{L}),
\end{equation}
where the vertical coordinate is scaled as
\begin{equation}
    Z(x_3)=\mathrm{Bu}\qty(\frac{x_3}{H}-\frac{1}{2}),
\end{equation}
and the normalisation constant $n$ is given by 
\begin{equation}
    n = \frac{1}{\mathrm{Bu}}\qty(\qty(\frac{\mathrm{Bu}}{2}-\tanh{\frac{\mathrm{Bu}}{2}})\qty(\coth\frac{\mathrm{Bu}}{2}-\frac{\mathrm{Bu}}{2}))^{\frac{1}{2}}.
\end{equation}

Finally, the initial particle distribution is determined by the transport map $T_0(\vb{x})$, which maps the physical coordinates to geostrophic coordinates. 
This map depends on the meridional velocity $v$, derived from the Exner pressure $\Pi$ via the geostrophic balance relation Eq.~\eqref{eq:balancerelations1}, where $\Pi$ is obtained from $\theta(\cdot,0)$ by the hydrostatic balance Eq.~\eqref{eq:balancerelations2}:
\begin{equation}
    T_0(\vb{x}) = \mqty(x_1 + \f^{-1}v(\vb{x}, 0) \\ \frac{g\theta(\vb{x},0)}{\f^2\theta_0}).
\end{equation}
The hydrostatic balance is integrated upwards from $\Pi=1$ at $x_3=0$, and $\sigma_0=\rho\theta$ follows from the equation of state. 
The initial particles are generated from $T_0$ as follows. 
The slice is cut into $N_x=2n$ columns of equal width, and each column into $N_z=n$ pieces of equal $\sigma_0$-mass, so that $N=2n^2$. 
Particle $i$ receives the $\sigma_0$-mass $m^i_0$ of its piece $P^i$ and is placed at the $\sigma_0$-weighted mean of $T_0$ over it, $\vb{z}^i_0=\frac{1}{m^i_0}\int_{P^i}T_0\,\dd\sigma_0$. 
For this initial condition, Eq.~\eqref{eq:RefPress} gives $\Pi_0=0.8602$. 
As a check, the optimal transport solution at $t=0$ reproduces the initial condition: the Exner pressure in each cell is within $7\times10^{-5}$ of its initial value, and the cell centroids are within about $1$\,km of the $\sigma_0$-barycentres of the pieces $P^i$.

\subsection{Validation: Single Seed Benchmark}\label{sec:singleseed}

To establish a foundational test case and validate the accuracy of our integral computations, we consider a first benchmark problem consisting of a single particle seed. 
Unlike the full simulation, this scenario admits an analytical solution, providing a rigorous ground truth for verifying our numerical implementation. 

For this specific benchmark, we adopt the following simplified set of non-dimensional parameters
\begin{equation}\label{eq:singleseedparams}
    N_{\mathrm{BV}}=c_p=\theta_0=s=R_d=p_0=\Pi_0=\f=g=1\,, \quad \gamma=2\,, \quad \kappa=\frac{1}{2}.
\end{equation}
The computational domain is a two-dimensional box, periodic in the horizontal direction, defined as
\begin{equation}
    \X=(-1,1]\times[0,1].
\end{equation}

In this setting, the optimal weight $w_*$ for the single Laguerre cell is determined solely by the mass constraint, which requires the total mass of the cell to integrate to one. 
This leads to the integral equation 
\begin{equation}\label{eq:singleseedweights}
    1=\int_{L^i_c(\vb*{w},\vb{z})}(f^*)'(w-c(\vb{x},\vb{z})) \,\dd \vb{x},
\end{equation}
for $\vb{x}=(x_1,x_2)^T$ and $\vb{z}=(z_1, z_2)^T$, where $x_2$ and $z_2$ denote vertical coordinates (see Section~\ref{sec:nondim}). 
With the parameters Eq.~\eqref{eq:singleseedparams}, $\lambda=1$ in Eq.~\eqref{eq:lambda} and $(f^*)'(r)=r_+$ by Eq.~\eqref{eq:fstar}. 
In this benchmark only, we keep the constant $-c_p\Pi_0=-1$ in the cost, $c(\vb{x},\vb{z})=\frac{1}{z_2}\qty(\frac12(x_1-z_1)^2+x_2)-1$, as in the original formulation. By the remark following Eq.~\eqref{eq:compcost2D}, this only shifts the optimal weight by $1$ and leaves the cell, $\sigma$, $C_1$ and $E_I$ unchanged. 
The density on the cell is $\sigma=\qty(w+1-\frac{1}{z_2}\qty(\frac12(x_1-z_1)^2+x_2))_+$, and its support is the region below the parabola $x_2=z_2(w+1)-\frac12(x_1-z_1)^2$.
Solving this equation yields the optimal weight $w_*$ as a function of the particle's vertical position $z_2$.

To solve for the optimal weight $w_*$, we first define the function $p(w,z_2)$
\begin{equation}
    \begin{split}
    p(w, z_2)&=\frac{1}{60 z_2}\left(3-32 \sqrt{2} \sqrt{w z_2+z_2-1}\right.\\
    &\quad-4 (w+1) z_2 \left(-16 \sqrt{2} \sqrt{w z_2+z_2-1}\right.\\
    &\left.\left.\quad\quad+(w+1) z_2 \left(8 \sqrt{2} \sqrt{w z_2+z_2-1}-15\right)+5\right)\right).
    \end{split}
\end{equation}
By solving Equation~\eqref{eq:singleseedweights}, we obtain the optimal weight $w_*$ as a piecewise function of the particle's vertical position, $z_2$
\begin{equation}\label{eq:wstar}
     w_*=\begin{cases}
        \frac{1}{6z_2}\qty(4-3z_2) & \text{if } z_2\geq \frac{5}{3}, \\
         \frac{1-6 z_2+6\sqrt{z_2-\frac{1}{45}}}{6 z_2} & \text{if } \frac{2}{15}\leq z_2\leq\frac{43}{60}, \\
         \frac{1}{z_2}\qty(\frac{15 z_2}{8\sqrt{2}})^{2/5}-1 & \text{if } 0< z_2\leq\frac{2}{15}, \\
          \text{The solution to } 1=p(w, z_2) &\text{if } \frac{43}{60}<z_2<\frac{5}{3}.
     \end{cases}
\end{equation}
The four regimes correspond to the position of the parabola bounding the support of $\sigma$ relative to the box $\X$. 
For $z_2\geq5/3$, $\sigma>0$ on all of $\X$. 
For $43/60<z_2<5/3$, the support touches the top boundary $x_2=1$ and the side walls. 
For $2/15\leq z_2\leq43/60$, it touches the side walls $x_1-z_1=\pm1$ but not the top. For $z_2<2/15$, it lies strictly inside the box.

The relationship between the optimal weight $w_*$ and the particle's height $z_2$ is visualized in Figure~\ref{fig:optimal_weight_vs_z2}. 

\begin{figure}
    \centering
    \begin{tikzpicture}
    \begin{axis}[
        title={Optimal Weight $w_*$ as a Function of Particle Height $z_2$},
        xlabel={$z_2$},
        ylabel={$w_*$},
        xmin=0, xmax=3, 
        ymin=-1, ymax=3,
        grid=major,
    ]
    
    \addplot[
        domain=2/15:43/60, samples=200,
        color=black, thick,
    ]
    { (1 - 6*x + 6*sqrt(x - 1/45)) / (6*x) };
    
    \addplot[
        domain=5/3:3, samples=100,
        color=black, thick,
    ]
    { (4 - 3*x) / (6*x) };

    \addplot[
        domain=0.03:2/15, samples=100,
        color=black, thick,
    ]
    { (15*x/(8*sqrt(2)))^(2/5)/x - 1 };

    \addplot[dashed, gray] coordinates {(0.13333,-1) (0.13333,3)};
    \addplot[dashed, gray] coordinates {(0.71667,-1) (0.71667,3)};
    \addplot[dashed, gray] coordinates {(1.66667,-1) (1.66667,3)};

    \addplot[
        color=black, thick, 
    ]
    table {numerical_data.txt};
    
    \end{axis}
    \end{tikzpicture}
    \caption{The optimal weight $w_*$ of the single seed as a function of its height $z_2$: the closed forms of Eq.~\eqref{eq:wstar}, and for $43/60<z_2<5/3$ the numerical solution of $1=p(w,z_2)$. Dashed lines mark the breakpoints $z_2=2/15$, $43/60$ and $5/3$ between the four regimes.}
    \label{fig:optimal_weight_vs_z2}
\end{figure}
   
With the optimal weight determined, we can analytically evaluate the derived quantities. 
The first component of the centroid of the Laguerre cell, $C_1$, simplifies elegantly due to the symmetry of the domain and the cost function.
With the change of variables, $u=x_1-z_1$,
\begin{align}
    C_1&=\int_{L^i_c(w_*,\vb{z})}x_1(f^*)'(w_*-c(\vb{x},\vb{z})) \,\dd \vb{x} =\int_0^1\int_{z_1-1}^{z_1+1}x_1(f^*)'(w_*-c(\vb{x},\vb{z})) \,\dd \vb{x} \\
    &=\int_0^1\int_{-1}^{1}\qty(u+z_1)(f^*)'\qty(w_*-\frac{1}{z_2}\qty(\frac{1}{2}u^2+x_2)+1) \,\dd u\dd x_2 \\
    &=\int_{L^i_c(w_*,\vb{z})}\underbrace{u(f^*)'\qty(w_*-\frac{1}{z_2}\qty(\frac{1}{2}u^2+x_2)+1)}_{\substack{\text{odd function in } u \\ \text{over symmetric domain } [-1, 1]}}\,\dd u \dd x_2\\
    &\quad+z_1\int_{L^i_c(w_*,\vb{z})}(f^*)'\qty(w_*-\frac{1}{z_2}\qty(\frac{1}{2}u^2+x_2)+1)\,\dd u \dd x_2 \\
    &= 0+z_1 = z_1.
\end{align}
The integral of the odd function over a symmetric domain vanishes, resulting in the first component of the centroid simply being the particle's horizontal position, $z_1$. 
Referring back to our non-dimensional ODEs (Eq.~\eqref{eq:nondimpartsys}), this equality immediately implies that the vertical velocity of the particle is zero. 
Consequently, through our exact algebraic mass-height relationship, the mass of the particle remains perfectly constant, which is exactly required for a single-seed system holding the total domain mass. 
Next, we evaluate the internal energy contribution, $E_I$, which is defined as
\begin{align}
    E_{I} &= \int_{L^i_c(w_*,\vb{z})}(f^*)'(w_*-c(\vb{x},\vb{z}))-(f^*)'(w_*-c(\vb{x},\vb{z}))^2 \,\dd \vb{x} \\
    &= 1-\int_{L^i_c(w_*,\vb{z})}(f^*)'(w_*-c(\vb{x},\vb{z}))^2\,\dd \vb{x}.
\end{align}
To express the internal energy as a function of $z_2$, we introduce the function $q(w_*,z_2)$ 
\begin{equation}
\begin{split}
    q(w_*,z_2)&=1-\frac{1}{420 z_2^2} \Biggl( 128 \sqrt{2} \sqrt{w_* z_2 + z_2 - 1} \\
    & -2 (w_*+1) z_2 \biggl( 3 \left(64 \sqrt{2} \sqrt{w_* z_2+z_2-1}-7\right) \\
    & \qquad +2 (w_*+1) z_2 \Bigl(-96 \sqrt{2} \sqrt{w_* z_2+z_2-1} \\
    & \qquad \qquad +2 (w_*+1) z_2 \left(16 \sqrt{2} \sqrt{w_* z_2+z_2-1}-35\right)+35\Bigr) \biggr) -5 \Biggr).
\end{split}
\end{equation}
Similar to the optimal weight, the internal energy $E_I$ is a piecewise function of $z_2$
\begin{equation}\label{eq:EIstar}
     E_I=\begin{cases}
        \frac{1}{2} - \frac{19}{90z_2^2} & \text{if } z_2\geq \frac{5}{3}, \\
        1- \frac{630 z_2 \sqrt{225 z_2-5}+28 \sqrt{225 z_2-5}-20}{14175 z_2^2} & \text{if } \frac{2}{15}\leq z_2\leq\frac{43}{60}, \\
         1-\frac{4}{7}\qty(w_*+1)=1-\frac{4}{7z_2}\qty(\frac{15z_2}{8\sqrt2})^{2/5} & \text{if } 0<z_2\leq\frac{2}{15}, \\
         q(w_*, z_2), \text{ where } w_* \text{ solves } 1=p(w, z_2)  &\text{if } \frac{43}{60}<z_2<\frac{5}{3}. \\
     \end{cases}
\end{equation}

Figure~\ref{fig:internal_energy_vs_z2} plots the internal energy $E_I$ as a function of the particle's height $z_2$. 

\begin{figure}
    \centering
    \begin{tikzpicture}
    \begin{axis}[
        title={Optimal Internal Energy $E_I$ as a Function of $z_2$},
        xlabel={$z_2$},
        ylabel={$E_I$},
        xmin=0, xmax=3,
        ymin=-3, ymax=1, 
        grid=major,
    ]
    
    \addplot[
        domain=5/3:3, samples=100,
        color=black, thick,
    ]
    { 0.5 - (19/(90*x^2)) };
    
    \addplot[
        domain=2/15:43/60, samples=200,
        color=black, thick,
    ]
    { 1 - ( (630*x*sqrt(225*x-5) + 28*sqrt(225*x-5) - 20) / (14175*x^2) ) };

    \addplot[
        domain=0.06:2/15, samples=100,
        color=black, thick,
    ]
    { 1 - 4/(7*x)*(15*x/(8*sqrt(2)))^(2/5) };

    \addplot[dashed, gray] coordinates {(0.13333,-3) (0.13333,1)};
    \addplot[dashed, gray] coordinates {(0.71667,-3) (0.71667,1)};
    \addplot[dashed, gray] coordinates {(1.66667,-3) (1.66667,1)};

    \addplot[
        color=black, thick,
    ]
    table {numerical_energy_data.txt};
    
    \end{axis}
    \end{tikzpicture}
    \caption{The internal energy $E_I$ of the single seed at the optimal weight, as a function of its height $z_2$: the closed forms of Eq.~\eqref{eq:EIstar}, and for $43/60<z_2<5/3$ the value $q(w_*,z_2)$ with $w_*$ computed numerically. Dashed lines mark the breakpoints $z_2=2/15$, $43/60$ and $5/3$.}
    \label{fig:internal_energy_vs_z2}
\end{figure}

This analysis simplifies the governing ODE system to a trivial trajectory
\begin{equation}
    \begin{cases}
    \displaystyle
        \dot{z}_1=-E_I(z_2), \\
    \displaystyle
        \dot{z}_2=0, 
    \end{cases}
\end{equation}
where the sign follows from Eq.~\eqref{eq:particle_system}: $\dot z_1=-\frac{sg}{\f\theta_0}\frac{c_p\theta_0}{g}E_I=-E_I$ for the parameters Eq.~\eqref{eq:singleseedparams}.
This implies that the particle moves horizontally with a constant velocity determined by its fixed height.
Given an initial condition $\overline{\vb{z}} = \vb{z}_0$, the trajectory is given by
\begin{equation}\label{eq:singleseedtraj}
    \vb{z}_t=\qty(-E_I(\overline{z}_2) t +\overline{z}_1)\vu{e}_1+\overline{z}_2\vu{e}_2.
\end{equation}
This analytical solution for the particle's path in the single seed benchmark provides a crucial tool for validating our numerical methods. 
By comparing the output of our numerical simulations against this exact solution, we confirmed that our algorithms for cell construction and integral evaluation are functioning correctly and producing accurate values.
Table~\ref{tab:singleseed} reports this comparison for twelve heights $z_2$ covering the four regimes of Eq.~\eqref{eq:wstar}. 
The closed forms Eq.~\eqref{eq:wstar}-\eqref{eq:EIstar} agree with an independent numerical quadrature of Eq.~\eqref{eq:singleseedweights} to $10^{-14}$. 
The internal energy computed by the scheme agrees with Eq.~\eqref{eq:EIstar} to about $10^{-12}$, the horizontal offset $C_1-z_1$ of the centroid vanishes to $5\times10^{-11}$, and the trajectory computed with AB2 follows Eq.~\eqref{eq:singleseedtraj} to $3\times10^{-11}$, which also confirms the sign in Eq.~\eqref{eq:singleseedtraj}. 
The error in the optimal weight, between $10^{-7}$ and $7\times10^{-6}$, comes from the chords that replace the parabolic bottom and top of the chart domain (Section~\ref{sec:Geom_Cost}). 
It decreases as the square of the strip width, from $1.3\times10^{-2}$ with $16$ strips to $3.2\times10^{-6}$ with $1024$ strips at $z_2=0.1$, while the error in $E_I$ decreases from $2.3\times10^{-5}$ to $1.2\times10^{-12}$.

\begin{table}[!ht]
    \centering
    \small
    \begin{tabular}{ccrrcccc}
    \toprule
    $z_2$ & case & \multicolumn{1}{c}{$w_*$} & \multicolumn{1}{c}{$E_I$} & $|w-w_*|$ & $|E_I^{\mathrm{num}}-E_I|$ & $|C_1-z_1|$ & trajectory \\
    \midrule
    0.05 & 3 & $\num{5.75480}$ & $\num{-2.85989}$ & \num{6.4e-6} & \num{2.9e-13} & \num{2.6e-13} & \num{1.9e-12} \\
    0.1 & 3 & $\num{3.45651}$ & $\num{-1.54657}$ & \num{3.2e-6} & \num{1.2e-12} & \num{1.6e-13} & \num{1.4e-12} \\
    $\frac{2}{15}$ & 2 & $\num{2.75000}$ & $\num{-1.14286}$ & \num{2.4e-6} & \num{4.7e-13} & \num{5.6e-17} & \num{9.3e-13} \\
    0.3 & 2 & $\num{1.31238}$ & $\num{-0.32905}$ & \num{1.1e-6} & \num{5.1e-14} & \num{3.8e-11} & \num{4.8e-12} \\
    0.5 & 2 & $\num{0.71576}$ & $\num{0.00210}$ & \num{6.4e-7} & \num{5.3e-14} & \num{4.3e-11} & \num{2.5e-11} \\
    $\frac{43}{60}$ & 2 & $\num{0.39535}$ & $\num{0.17948}$ & \num{4.4e-7} & \num{3.9e-14} & \num{4.2e-11} & \num{5.2e-12} \\
    0.8 & 4 & $\num{0.31101}$ & $\num{0.22590}$ & \num{4.0e-7} & \num{5.9e-14} & \num{4.1e-11} & \num{7.3e-12} \\
    1 & 4 & $\num{0.15935}$ & $\num{0.30584}$ & \num{3.2e-7} & \num{5.2e-14} & \num{3.9e-11} & \num{8.5e-12} \\
    1.3 & 4 & $\num{0.01197}$ & $\num{0.37689}$ & \num{2.4e-7} & \num{2.7e-14} & \num{3.7e-11} & \num{4.6e-12} \\
    $\frac{5}{3}$ & 1 & $\num{-0.10000}$ & $\num{0.42400}$ & \num{1.9e-7} & \num{3.6e-14} & \num{3.6e-11} & \num{8.6e-12} \\
    2 & 1 & $\num{-0.16667}$ & $\num{0.44722}$ & \num{1.6e-7} & \num{2.5e-14} & \num{3.0e-11} & \num{3.9e-12} \\
    3 & 1 & $\num{-0.27778}$ & $\num{0.47654}$ & \num{1.1e-7} & \num{1.2e-14} & \num{2.0e-11} & \num{3.5e-12} \\
    \bottomrule
    \end{tabular}
    \caption{Single-seed benchmark. Closed-form optimal weight $w_*$ and internal energy $E_I$ from Eq.~\eqref{eq:wstar}-\eqref{eq:EIstar}, and the errors of the scheme with $1024$ strips: optimal weight, internal energy, horizontal offset of the centroid, and maximum error of the AB2 trajectory against Eq.~\eqref{eq:singleseedtraj}. The cases are the four regimes of Eq.~\eqref{eq:wstar}, in the order listed there.}
    \label{tab:singleseed}
\end{table}

\subsection{Simulation Results}\label{sec:results}

\subsubsection{Full Simulation Dynamics}

The evolution of the compressible frontogenesis test case is visualized in Figure~\ref{fig:grid_evolution}, at days 2, 4, 7 and 11 as in \cite[Figures~1 and~2]{Cotter:2025}. 
Its columns show the seeds in geostrophic space, the out-of-slice velocity $v^i=\f(z^i_1-C^i_1)$ and the potential temperature perturbation of each cell, and the departure of the source density $\sigma$ from its horizontal mean at each height, evaluated pointwise from the dual formula $\sigma=(f^*)'\big(\max_i(w^i_*-c(\cdot,\vb{z}^i))\big)$. 
Figure~\ref{fig:grid_evolution_inc} shows the same fields, except $\sigma$, for the incompressible SG model, computed with the scheme of \cite{Egan:2022, Lavier:2024} from the incompressible version of the same initial condition (with the linear background profile $\overline{\theta}=\theta_0(1+N_{\mathrm{BV}}^2(x_3-H/2)/g)$) and bred in the same way.
The simulation uses $N=10368$ particles ($N_x=144$, $N_z=72$) and the timestep $h=1800$\,s.\footnote{It was run on a workstation with an AMD Ryzen 9 3950X processor (16 cores, 32 threads), with a wall-clock time of $8.4$\,h for the $97$ steps of the breeding phase and the $1200$ steps of the 25 days that follow, i.e.\ $23.3$\,s per timestep on average.} 
The cost of a step follows the sharpness of the front: over the first 2.5 days, a step takes on average $6.6$ Newton iterations and $3.4$\,s, and once the front has formed (after day 7.5), $30$ to $40$ Newton iterations (at most $79$) and $24$ to $38$\,s. 
For comparison, the incompressible SG code of \cite{Egan:2022, Lavier:2024}, run on the same machine with the same $N$ and $h$, took $11.7$\,s per timestep ($4.3$\,h for $1320$ steps, breeding included), so a step of the compressible scheme costs about twice as much, although that code stops at a looser tolerance, $10^{-5}$ of the mean cell area.

Qualitatively, our results capture the standard lifecycle of the baroclinic instability, exhibiting excellent agreement with the structural evolution reported in the reference finite-element solutions.
The reference Exner pressure $\Pi_0$ only fixes a horizontal frame of reference. 
In \cite{Cotter:2025}, $\Pi_0=0.864$ is tuned to minimise the horizontal drift of the front, whereas we use the $\sigma$-weighted average Eq.~\eqref{eq:RefPress}, $\Pi_0=0.8602$, without tuning.
By Lemma~\ref{lem:Pi0}, changing $\Pi_0$ by $\Delta\Pi_0$ translates the whole solution horizontally at the constant speed $-\frac{sc_p}{\f}\Delta\Pi_0$, which is about $30$\,m\,s$^{-1}$ per unit change of $\Pi_0$ for the parameters of Table~\ref{tab:params}. 
The predicted drift relative to \cite{Cotter:2025} is therefore only $0.114$\,m\,s$^{-1}$, or $139$\,km at day 12, counting the breeding phase.
After the front forms, small numerical differences between the two runs grow, as expected from the sensitivity of the semi-geostrophic equations \cite{Lavier:2024},the worst parcel is $100$\,km from its predicted position at day 10, while the RMS velocities agree to a relative $1.4\times10^{-4}$ and the energy errors to $5\times10^{-7}$ of the peak kinetic energy.

The dynamics proceed through distinct phases.
The system begins with a broad, gentle wave structure characteristic of linear normal mode. 
The particle distribution remains relatively uniform, though initial displacements are visible.
The instability then matures into a sharp frontal discontinuity.
The Lagrangian nature of our scheme is most evident here.
The particles do not cluster at the front: the area of a cell is set by its mass, which only changes with its potential temperature, and by $\sigma$, which is continuous and departs from its horizontal mean by a few percent at most. 
The cells adapt their shape instead: along the front they become thin and elongated, so that $v$ and $\theta'$ jump across a narrow strip of cells (day 7 of Figure~\ref{fig:grid_evolution}). 
In geostrophic coordinates, the seed cloud is sheared and folded, into a lens and a thin filament by day 7, and by day 11 into one compact patch, wrapped across the periodic boundary, with thin filaments around it. 
By Eq.~\eqref{eq:alphaPDE}, the measure $\alpha_t/y_3$ is transported by the geostrophic velocity $\vb{w}$. Unlike the incompressible case, $\vb{w}$ is not divergence-free, but $\alpha_t$ itself only changes through the factor $y_3$.

The last column of Figure~\ref{fig:grid_evolution} shows the source density $\sigma=\rho\theta$ as a departure from its horizontal mean at each height. 
By the equation of state Eq.~\eqref{eq:stateeq}, $\sigma=(p_0/R_d)\,\Pi^{1/(\gamma-1)}$ is a function of the Exner pressure alone, so this departure is a pressure anomaly, $\delta\sigma/\overline{\sigma}\approx\delta p/(\gamma\overline{p})$. 
By the geostrophic balance Eq.~\eqref{eq:balancerelations1}, $\partial_{x_1}\sigma=\f\sigma v/\big((\gamma-1)c_p\theta\Pi\big)$ has the sign of $v$. 
At day 2, the anomaly is a broad wave, between $-0.33\%$ and $+0.30\%$, whose trough tilts towards negative $x_1$ with height, by about $500$\,km over the depth of the slice, as for a growing Eady wave. 
The wave amplifies to between $-1.05\%$ and $+0.76\%$ at day 4. 
Once the front has formed (day 7), the trough lies along the front, where $v$ jumps from negative to positive values. It is deeper ($-3.4\%$) than the broad ridge to the east of it ($+1.8\%$), where $v$ changes sign again. 
Since $\sigma$ is continuous, the front appears in this column as a kink rather than a jump. By day 11, the trough and the ridge are almost vertical, with anomalies of $-2.2\%$ and $+1.4\%$.

\begin{figure}[tbp]
    \centering
    \includegraphics[width=\linewidth,height=0.69\textheight,keepaspectratio]{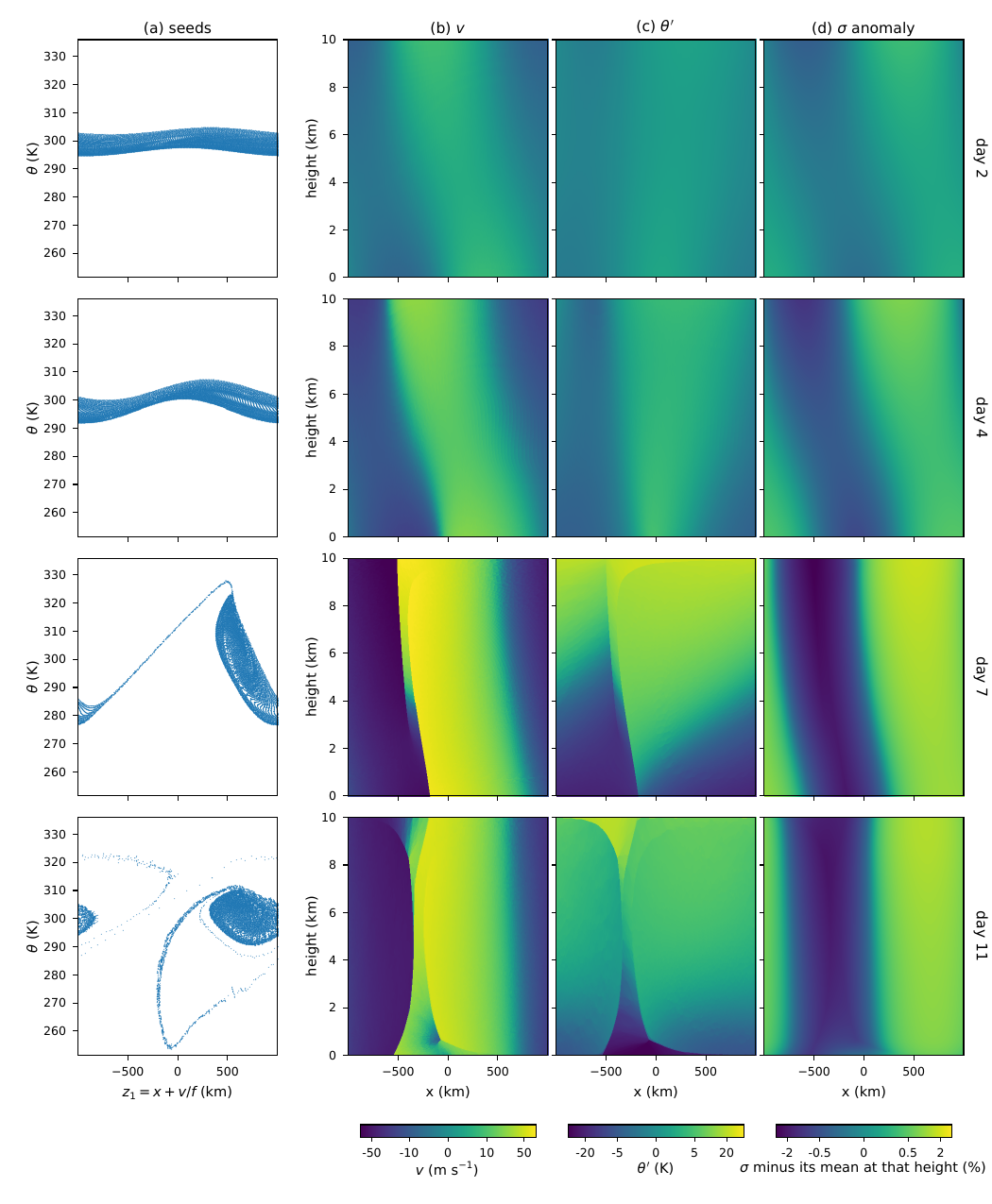}
    \caption{Compressible frontogenesis at days 2, 4, 7 and 11 after the breeding phase, $N=10368$, $h=1800$\,s. Columns: the seeds $\vb{z}^i$ in geostrophic space, with $z^i_3$ shown as the potential temperature; the out-of-slice velocity $v$ of each cell; the potential temperature perturbation $\theta'$ (K), the parcel temperature minus the background profile at the height of its centroid; and the source density $\sigma=\rho\theta$ as a departure from its horizontal mean at the same height (\% of that mean). The cells are drawn in the physical slice, and $\sigma$ is evaluated at every pixel from the global dual formula, so it varies within the cells and is continuous across them. Each column has one colour scale for all the days, shared with Figure~\ref{fig:grid_evolution_inc}; a few cells at the front exceed its limits, the $99.9$th percentiles, and take the end colours. The scales are close to linear for $|v|<4$\,m\,s$^{-1}$, $|\theta'|<2$\,K and $|\sigma$ anomaly$|<0.16\%$, and logarithmic beyond, so that the weak early perturbations remain visible.}
    \label{fig:grid_evolution}
\end{figure}

\begin{figure}[tbp]
    \centering
    \includegraphics[width=\linewidth,height=0.72\textheight,keepaspectratio]{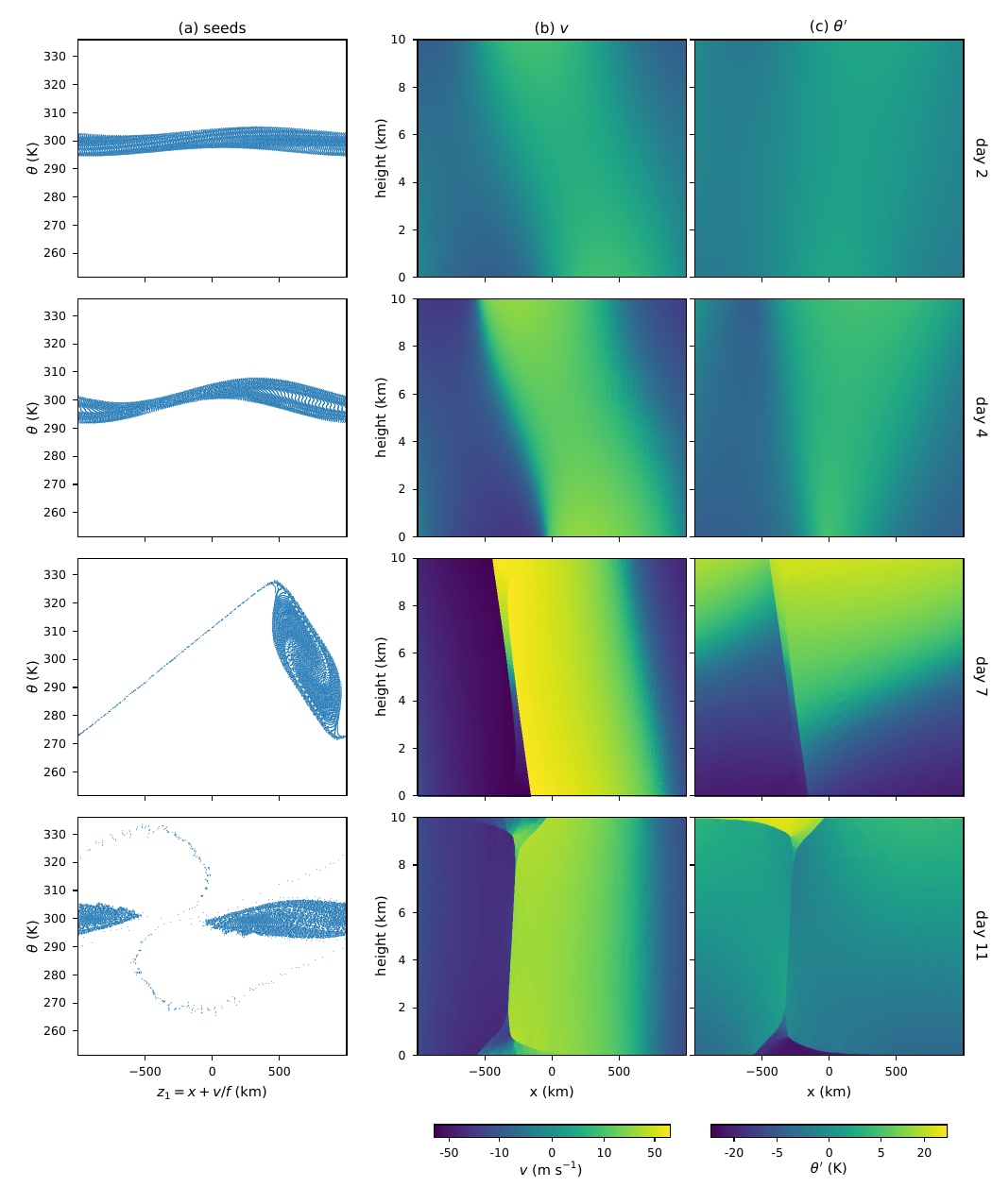}
    \caption{The incompressible SG model at the same days, with the same initial perturbation, breeding and resolution: seeds in geostrophic space, out-of-slice velocity $v$ and potential temperature perturbation $\theta'$, on the colour scales of Figure~\ref{fig:grid_evolution}.}
    \label{fig:grid_evolution_inc}
\end{figure}

\subsubsection{Comparison of the two models}\label{sec:rmsv}

Figure~\ref{fig:rmsv}(a) shows the root mean square of the out-of-slice velocity over the particles, $\big(\frac1N\sum_i (v^i)^2\big)^{1/2}$, for both models.
The particles have nearly equal masses, so this is close to a mass-weighted average.
Both models grow at the same rate during the linear phase. 
The compressible model reaches $18.2$\,m\,s$^{-1}$ at day 5 and its first maximum, $45.6$\,m\,s$^{-1}$, at day 7.9, against $18.5$\,m\,s$^{-1}$ and $42.0$\,m\,s$^{-1}$ at day 7.7 in \cite[Figure~3]{Cotter:2025} ($43.1$\,m\,s$^{-1}$ at day 7.8 in their high-resolution run). 
Over days 0 to 5 our compressible RMS velocity grows exponentially at a rate of $0.53$ per day, against $0.52$ in \cite{Cotter:2025}, and stays within $7\%$ of theirs. 
The incompressible model peaks slightly earlier and higher ($47.8$\,m\,s$^{-1}$ at day 7.7).

After the first front, the two models differ. 
The incompressible model repeats almost identical life cycles, with maxima of $47$--$48$\,m\,s$^{-1}$ at days 7.7, 14.6 and 21.6 and minima of about $21$\,m\,s$^{-1}$. 
In the compressible model the oscillation is much weaker, between about $31$ and $41$\,m\,s$^{-1}$. This is not dissipation: the energy error stays below $6\times10^{-4}$ of the peak kinetic energy (Section~\ref{sec:conservation}), whereas \cite{Cotter:2025} attribute the energy loss after the first front in their model to numerical dissipation. 
Figure~\ref{fig:rmsv}(b,c) shows the RMS velocity over the particles whose centroid lies in the lower and in the upper half of the domain. 
In the incompressible model the two halves agree to within $2\%$ from day 3 on, as they must: the incompressible Eady problem, initial condition included, is invariant under the rotation $(x_1,x_3)\mapsto(-x_1,H-x_3)$ combined with $(v,\theta')\mapsto(-v,-\theta')$, which exchanges the two halves. 
The larger differences before day 3 ($16\%$ at day 0) come from the discretisation and do not grow. 
The compressible model breaks this symmetry, since $\sigma$ decreases by a factor of about 2.7 from the ground to the lid. 
The upper half peaks first and cycles faster, with maxima at days 7.3, 13.4, 18.3 and 22.6 (mean period 5.1 days), than the lower half, with maxima at days 8.0, 13.8 and 21.8 (mean period 6.9 days).
The two fronts drift out of phase, so the whole-depth RMS velocity looks damped. 
This behaviour is converged in $N$ over the first life cycle. 
The resolution study runs for $12.5$ days from the initial condition, that is, up to about day $10.5$. 
Over that time the RMS velocities differ from those at $N=10368$ by at most $4\%$ at $N=288$, $1.3\%$ at $N=648$ and $0.8\%$ from $N=1152$ on, in both models, and at every $N$ the compressible upper half peaks $0.6$ to $0.75$ days before the lower half. 
The later cycles were computed at $N=10368$ only.

\begin{figure}[!ht]
    \centering
    \includegraphics[width=\linewidth]{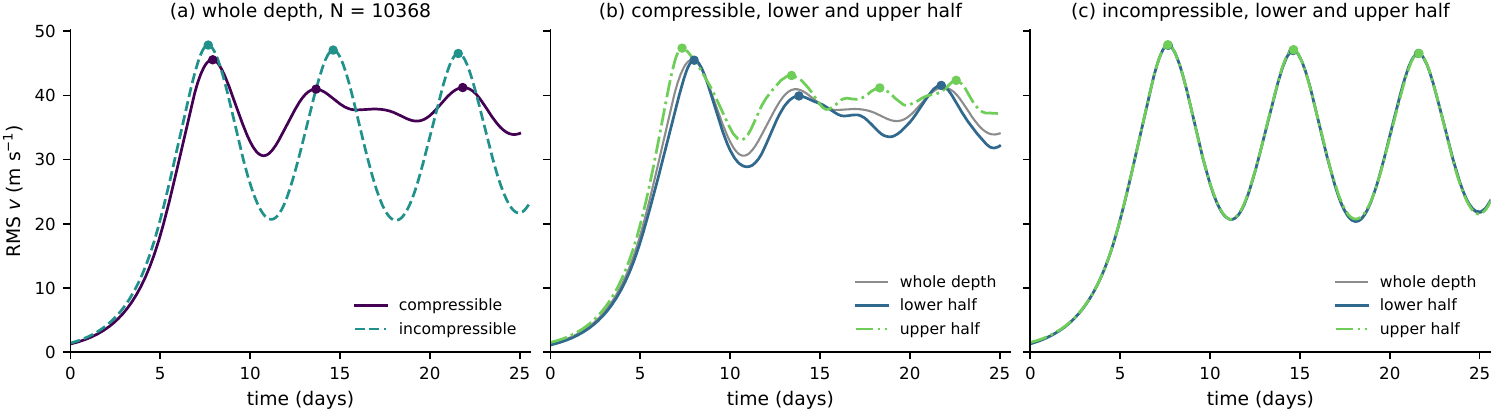}
    \caption{RMS out-of-slice velocity of the main runs ($N=10368$, times after the breeding phase). (a)~Compressible (solid) and incompressible (dashed) models over the whole depth. (b,c)~Over the particles whose centroid lies in the lower and in the upper half of the domain, for each model, with the whole-depth curve in grey; dots mark the maxima.}
    \label{fig:rmsv}
\end{figure}

\subsubsection{Conservation Properties}\label{sec:conservation}

To quantify the conservation properties of the numerical scheme, we monitor the energy and the total mass. Since the energy is only defined up to an additive constant\footnote{compare Eq.~\eqref{eq:energy_rewritten}, and the incompressible energy of \cite{Egan:2022}, which is measured relative to the basic state}, its relative error depends on that constant. 
We therefore measure the energy error against the peak kinetic energy of the out-of-slice motion,
\begin{equation}\label{eq:errordefn}
    e_E(t)=\frac{|E_g(t)-E_g(0)|}{\max_{s}K(s)},\qquad K(t)=\frac12\sum_{i=1}^N\frac{m^i}{\theta^i}\,(v^i)^2,\quad \theta^i=\frac{z^i_3}{\lambda},
\end{equation}
where $m^i/\theta^i$ is the mass of air in the cell (in the incompressible model, $K=\frac12\sum_i m^i(v^i)^2$ with $m^i$ the cell areas), and the error in the total mass by $e_m(t)=\big|\sum_i m^i(t)-\sum_i m^i(0)\big|/\sum_i m^i(0)$.

Figure~\ref{fig:energy} shows $e_E$ for the main runs of both models and for every run of the resolution study. 
While the continuous semi-geostrophic system conserves the energy, our fully discrete scheme introduces a temporal discretisation error through the explicit Adams--Bashforth integration. 
The error grows while the front forms and then stays bounded: over 25 days, $e_E\leq6.0\times10^{-4}$ for the compressible and $5.2\times10^{-4}$ for the incompressible main run ; relative to the total energy, the compressible error is $2.8\times10^{-6}$. 
The maximum is about $5\times10^{-4}$ for every $N$ from $288$ to $10368$ in both models, so the energy error comes from the time discretisation and not from the spatial one. 
Doubling the timestep multiplies it by $4.0$ to $4.1$ (Figure~\ref{fig:energy}(c): $3.0\times10^{-5}$, $1.2\times10^{-4}$, $5.0\times10^{-4}$ and $2.1\times10^{-3}$ for $h=450$, $900$, $1800$ and $3600$\,s at $N=2592$), as expected from the second order of AB2. 
The total mass, which the time-continuous system conserves exactly, changes by at most $4\times10^{-10}$ of its initial value over the 25 days of the compressible main run (at most $9\times10^{-10}$ in any compressible run), and the total cell area of the incompressible runs by at most $1.2\times10^{-12}$.

\begin{figure}[!ht]
    \centering
    \includegraphics[width=\linewidth]{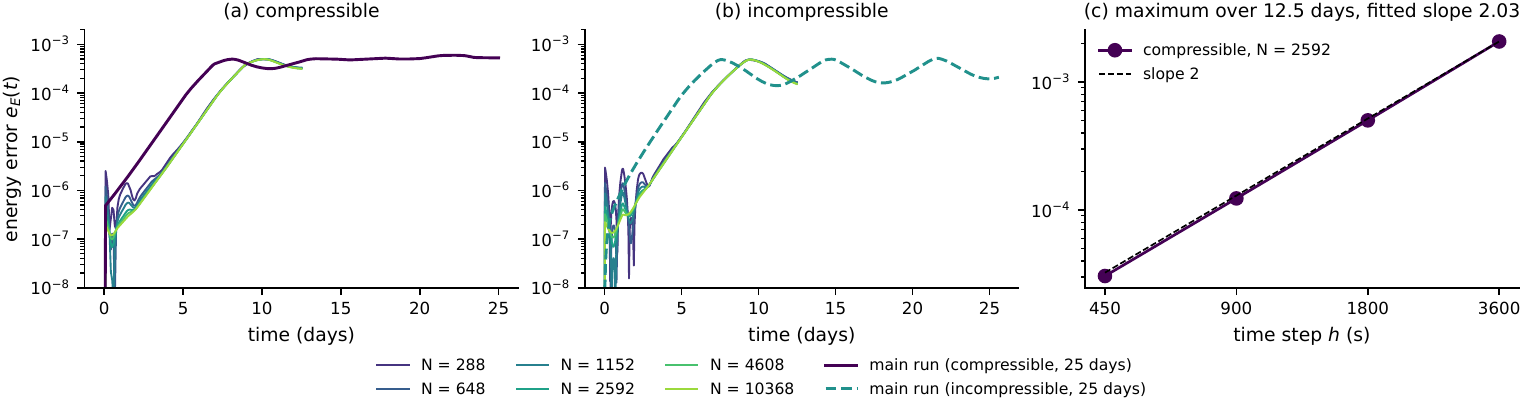}
    \caption{Energy error $e_E(t)$ Eq.~\eqref{eq:errordefn} of (a)~the compressible and (b)~the incompressible model: main runs ($N=10368$, 25 days after the breeding phase) and resolution study ($N=288$ to $10368$, 12.5 days from the initial condition). (c)~Maximum of $e_E$ over 12.5 days against the timestep, at $N=2592$, with a line of slope 2.}
    \label{fig:energy}
\end{figure}

\subsubsection{Convergence Analysis}\label{sec:convergence}

To assess the numerical convergence of the proposed scheme, we quantify the error in the particle distribution using the Wasserstein-2 distance between the discrete measures $\mu_N(t)=\sum_{i=1}^Nm^i\delta_{\vb{z}^i(t)}$, normalised so that $\sum_im^i=1$. 
Before measuring distances, we rescale the geostrophic coordinates to $\hat z_1=z_1/(2L)$ and $\hat z_3=z_3/\Delta_3$, where $\Delta_3=\max_iz^i_3(0)-\min_iz^i_3(0)$ for the finest run, so that both vary over a range of order one, and we use the periodic distance in $\hat z_1$. 
The same scaling is used for every run of a study, so it does not affect the rates. 
The distances are computed exactly, by solving the discrete optimal transport problem with the network simplex algorithm of the POT library \cite{Flamary:2021}. 
Because the semi-geostrophic equations are sensitive to perturbations of the initial conditions \cite{Lavier:2024}, convergence rates can be expected to deteriorate once fronts have formed. 
We therefore report the errors every two days up to day 12, which covers the formation of the first front.

\paragraph{Number of particles.} The resolution study uses $N=2n^2$ particles with $n=12,18,24,36,48,72$, i.e.\ $N=288$ to $10368$, and $h=1800$\,s, for 12.5 days from the initial condition, for both models. 
Comparing every run with a single fine reference is not reliable here: an initial particle grid that is nested in the reference grid is much closer to it than one that is not, which makes the error against a reference decrease unevenly. 
We therefore use doubling pairs, $e_N(t)=W_2(\mu_N(t),\mu_{4N}(t))$, which need no reference solution. 
If $W_2(\mu_N(t),\mu(t))\approx C\delta^p$ with $\delta=N^{-1/2}$, the triangle inequality gives $(1-2^{-p})C\delta^p\lesssim e_N(t)\lesssim(1+2^{-p})C\delta^p$, so the slope of $\log e_N$ against $\log\delta$ estimates $p$.

Figure~\ref{fig:particleerror} shows $e_N$ for both models. 
At day 0, $W_2$ is exactly the error made by sampling the initial condition with $N$ particles, and it decreases at the rate $p=1.00$, i.e.\ as $N^{-\frac12}$. 
This is the optimal quantization rate for absolutely continuous measures in two dimensions: as established in \cite{Kloeckner:2012}, the optimal approximation of such a measure by $N$ points satisfies $W_p\sim N^{-\frac{1}{d}}$, i.e.\ $W_2\sim N^{-\frac12}$ for $d=2$. 
The dynamics preserve this rate while the front forms: for the compressible model, $p=1.01$, $1.01$, $1.02$, $1.14$ and $1.04$ at days 2, 4, 6, 8 and 10, and the incompressible model behaves in a similar way ($1.00$, $1.01$, $1.09$, $1.09$ and $1.15$). 
At day 12, just after the first front, the rate falls to $0.71$ in the compressible and $0.87$ in the incompressible model, as the frontal dynamics amplify the differences between the runs.

\paragraph{Timestep.} The timestep study uses $N=2592$ and $h=450$, $900$, $1800$ and $3600$\,s, for 12.5 days from the initial condition. 
Since the particles are the same in every run, we measure the error of the trajectories directly, against the run with $h_{\mathrm{ref}}=450$\,s,
\begin{equation}\label{eq:dtnorm}
    \mathrm{Error}(t) = \Big(\sum_{i=1}^N m^i\, \big|\hat{\vb{z}}^i(t) - \hat{\vb{z}}^i_{\mathrm{ref}}(t)\big|^2\Big)^{1/2},
\end{equation}
with the rescaled coordinates and the periodic distance in $\hat z_1$.
We use this trajectory error rather than $W_2$ because, in the compressible model, the masses $m^i=m^i_0z^i_3/z^i_{3,0}$ follow the potential temperature of the particles and therefore differ slightly between runs with different timesteps; $W_2$ then also measures the transport of this mass difference between neighbouring particles, and decays more slowly than the trajectory error.

Since the reference run has its own error, we fit the order $p$ with that error included: $\mathrm{Error}\approx C(h^p-h_{\mathrm{ref}}^p)$ if the errors of all runs point in the same direction, and $\mathrm{Error}\approx C(h^{2p}+h_{\mathrm{ref}}^{2p})^{1/2}$ if they are independent. 
Figure~\ref{fig:timesteperror} shows the error and the fitted order. 
Before the front forms (days 2 to 6), the first model fits the errors with $p=2.00$ and a root-mean-square misfit of at most $2\times10^{-3}$ in $\log\mathrm{Error}$: the trajectories converge at the second order of AB2, and the error ratios between successive timesteps are $5.0$ and $4.2$, as $C(h^2-h_{\mathrm{ref}}^2)$ predicts. 
After the front has formed (days 8 to 12), the second model fits better, with $p=0.96$ to $0.97$. 
This loss of order coincides with the sensitivity seen in the $\Pi_0$ test of Section~\ref{sec:results}, where small numerical differences between two runs grow from $25$\,m at day 4 to $100$\,km at day 10 while the RMS velocity and the energy are unaffected. 
We attribute it to the loss of smoothness of the velocity field at the front, where the cells become thin and their neighbours change from one step to the next.
As for the internal energy $E_I$ of the single-seed benchmark, which is only piecewise smooth, this introduces discontinuities in the derivatives of the velocity field, which degrade the order of the ODE solver. 
The energy error stays second order throughout (Figure~\ref{fig:energy}(c)), and at $h=1800$\,s the timestep error in $W_2$ stays below the spatial error at $N=2592$: it is $0.04$, $0.37$, $0.59$ and $0.53$ of $W_2(\mu_{2592},\mu_{10368})$ at days 6, 8, 10 and 12.

\begin{figure}[!ht]
    \centering
    \includegraphics[width=\linewidth]{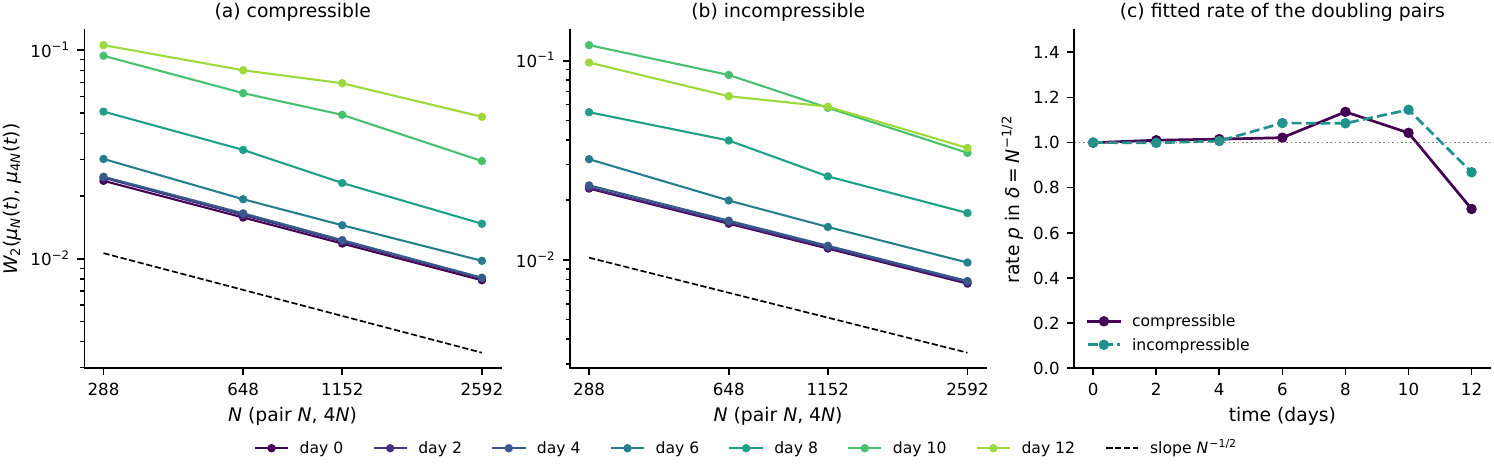}
    \caption{Convergence in the number of particles. $W_2$ distance between the runs with $N$ and $4N$ particles against $N$, at days 0 to 12, for (a)~the compressible and (b)~the incompressible model ($h=1800$\,s, no breeding). The dashed line has slope $N^{-\frac12}$, the optimal quantization rate. (c)~Rate $p$ in $\delta=N^{-1/2}$ fitted to the doubling pairs, against time.}
    \label{fig:particleerror}
\end{figure}

\begin{figure}[!ht]
    \centering
    \includegraphics[width=\linewidth]{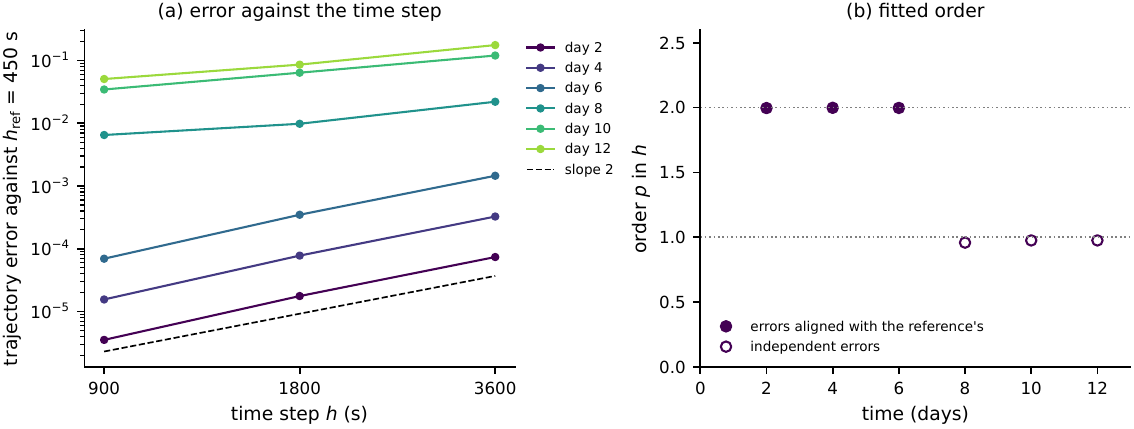}
    \caption{Convergence in the timestep. (a)~Trajectory error Eq.~\eqref{eq:dtnorm} against $h$ at days 2 to 12, for $N=2592$, measured against the run with $h_{\mathrm{ref}}=450$\,s, with a line of slope 2. (b)~Order $p$ fitted with the error of the reference run included. Filled: errors aligned with that of the reference, $\mathrm{Error}\approx C(h^p-h_{\mathrm{ref}}^p)$. Open: independent errors, $\mathrm{Error}\approx C(h^{2p}+h_{\mathrm{ref}}^{2p})^{1/2}$.}
    \label{fig:timesteperror}
\end{figure}

\FloatBarrier
\section*{Code availability}
The code used for the simulations, the error analysis and the figures of this paper is available at \url{https://github.com/thelavier/2DCompressibleSG} (branch \texttt{revision}). 
It builds on our forks of \texttt{pysdot} and \texttt{sdot} (\url{https://github.com/thelavier/pysdot}, \url{https://github.com/thelavier/sdot}), at the commits pinned in that repository. 
The incompressible simulations use the code of \cite{Lavier:2024}, \url{https://github.com/thelavier/2DIncompressibleSG}.

\section*{Acknowledgements} 
We thank Colin Cotter and Hiroe Yamazaki for stimulating discussions related to this work and sharing early drafts of their work. 
We thank Quentin Mérigot and Hugo Leclerc for their support in implementing the code.
We thank David Bourne for his invaluable advice and suggestions throughout the process.
TPL was supported by The Maxwell Institute Graduate School in Modelling, Analysis, and Computation, a Centre for Doctoral Training funded by the EPSRC (grant EP/S023291/1), the Scottish Funding Council, Heriot-Watt University and the University of Edinburgh. 
BP gratefully acknowledges the support of the EPSRC via the grant EP/P011543/1.

\bibliographystyle{abbrv}
\bibliography{references}

\newpage
\appendix

\section{Integral Reductions}\label{appendix:Integral_Reductions}

In the following we will always parametrize a path between two vertices $\vb{p}^j$ and $\vb{p}^{j+1}$ with $\upgamma:\R\to\R^d$ given by
\begin{equation}
    \upgamma(s)=s\vb{p}^{j+1}+(1-s)\vb{p}^{j},
\end{equation}
where $d$ is the dimension we are working in. 
In two dimensions we have the following cost function
\begin{equation}
    c(\vb{x},\vb{y}) = \frac{1}{y_2}\qty(\frac{\f^2}{2}\qty(x_1-y_1)^2+g x_2), 
\end{equation}
and 
\begin{equation}
    \Phi(\vb{p};\vb{y})\textrm{ with }\vb{y}=\mqty[0 \\ 1],
\end{equation}
so that, writing $\Phi(\vb{p}):=\Phi(\vb{p};\vb{y})$ for this fixed reference point,
\begin{equation}
    \Phi(\vb{p})=\mqty[\displaystyle\frac{p_1}{\f^2} \\ \displaystyle\frac{1}{g}\mqty(p_2 - \frac{p_1^2}{2\f^2})],
\end{equation}
and
\begin{equation}
    D\Phi=\mqty(\f^{-2} & 0 \\ -\f^{-2}p_1 & g^{-1})\textrm{ and }\abs{\det D\Phi}=\frac{1}{\f^2g}.
\end{equation}
In two dimensions $\vb{n}$ is the outward unit normal vector to the edge $[\vb{p}^j,\vb{p}^{j+1}]$ of the polygon $\Phi^{-1}(\Lag{i})$, in the chart coordinates, and $K^i$ is the number of vertices of the $i$th cell, with $\vb{p}^{K^i+1}:=\vb{p}^1$.

The cells are polygons in the chart because, in chart coordinates, the cost differences reduce to those of a power diagram, a Laguerre diagram for the quadratic cost. 
Indeed, by the proof of Lemma~\ref{lem:2DLoeper} with $\vb{y}=(0,1)$, $c(\Phi(\vb{p}),\vb{z}^i)=\frac{\f^2}{2z^i_2}(z^i_1)^2+\frac{p_2}{z^i_2}-\frac{p_1 z^i_1}{z^i_2}$ is affine in $\vb{p}$. 
Any affine function $\vb{p}\mapsto\vb{a}\cdot\vb{p}+b$ can be written as $\norm{\vb{p}-\vb{y}}^2-\norm{\vb{p}}^2-\psi$ with $\vb{y}=-\vb{a}/2$ and $\psi=\norm{\vb{y}}^2-b$. 
Applying this to $c(\Phi(\vb{p}),\vb{z}^i)-w^i$ gives the following ``power-diagram'' seeds $\vb{y}^i$ and weights $\psi^i$ in the chart.
By letting
\begin{align}
    \vb{y}^i\qty(\vb{z}^i) &= \vb{y}^i=\frac{1}{2z^i_2}\mqty[z_1^i \\ -1], \\
    \psi^i\qty(\vb{z}^i,w^i) &= w^i + \qty(\frac{z_1^i}{2z^i_2})^2 + \qty(\frac{1}{2z^i_2})^2 - \frac{\f^2}{2z^i_2}\qty(z^i_1)^2,
\end{align}
then
\begin{align*}
    c\qty(\Phi(\vb{p}),\vb{z}^i) - w^i - c\qty(\Phi(\vb{p}),\vb{z}^j) + w^j & = \norm{\vb{p}-\vb{y}^i}^2-\norm{\vb{p}}^2 - \psi^i - \norm{\vb{p}-\vb{y}^j}^2 + \norm{\vb{p}}^2 + \psi^j \\
    &= \norm{\vb{p} - \vb{y}^i}^2 - \psi^i - \norm{\vb{p}-\vb{y}^j}^2 + \psi^j.
\end{align*}
Thus $\Phi^{-1}(\Lag{i})$ is the $i$-th cell of the power diagram with seeds $\vb{y}^i$ and weights $\psi^i$, intersected with $\Phi^{-1}(\X)$, and it can be computed with standard power-diagram algorithms, i.e.~half-plane clipping.

\emph{On the choice of $f$.} The reductions in Sections~\ref{app:grad}-\ref{app:centroid} hold for any convex internal energy density $f$, since they only use the divergence theorem in the chart. 
The reduction of Section~\ref{app:EI} uses the explicit form $f(s)=\frac{\kappa}{\lambda}s^\gamma$, through $f^*$ in Eq.~\eqref{eq:fstar}. 
This is also where the consistency between $f$ and the energy Eq.~\eqref{eq:2Dcompenergy1} is needed, since $\Pi=(R_d\sigma/p_0)^{\gamma-1}$ must match $f'$.

\subsection{Gradient}\label{app:grad}

In the two dimensional problem the gradient is given by 
\begin{equation}
    \pdv{\G}{w^i}=m^i-\int_{\Lag{i}}(f^*)'(w^i-c(\vb{x},\vb{z}^i))\,\dd\H^2.
\end{equation}

We can reduce the integral to an evaluation at the vertices as by using the divergence theorem with 
\begin{equation}
    \vb{F}=\mqty[0 \\ -z_2^if^*(w^i-c(\Phi(\vb{p}),\vb{z}^i))],
\end{equation}
as follows
\begin{align*}
    &\int_{\Lag{i}}(f^*)'(w^i-c(\vb{x},\vb{z}^i))\,\dd\H^2 \\
    &= \frac{1}{\f^2g}\int_{\Phi^{-1}(\Lag{i})}(f^*)'(w^i-c(\Phi(\vb{p}),\vb{z}^i)) \,\dd\H^2 \\
    &= \frac{1}{\f^2g}\int_{\partial\Phi^{-1}(\Lag{i})}\vb{F}\cdot\vb{n}\,\dd\H^1 \\ 
    &=-\sum_{j=1}^{K^i}\frac{z_2^in_2}{\f^2g}\norm{\vb{p}^{j+1}-\vb{p}^j}\int_0^1f^*(w^i-c(\Phi(\upgamma(s)),\vb{z}^i))\,\dd s.
\end{align*}

\subsection{Hessian}

In the two dimensional problem the Hessian is given by 
\begin{align}
    \pdv{\G}{w^i}{w^j} \qty(\vb*{w},\vb{z})&= 
    \int_{\Lag{i}\cap \Lag{j}} \frac{1}{\norm{\nabla_{\vb{x}} u^{ij}(\vb{x})}}(f^*)'(w^i-c(\vb{x},\vb{z}^i)) \,\dd\H^1 \\
    \begin{split}
    \pdv{\G}{w^i}{w^i} \qty(\vb*{w}, \vb{z}) &= -\int_{\Lag{i}} \qty(f^*)''(w^i-c(\vb{x},\vb{z}^i)) \,\dd\H^2 \\
    &\quad-\sum_{i\neq j} \int_{\Lag{i}\cap \Lag{j}} \frac{1}{\norm{\nabla_{\vb{x}} u^{ij}(\vb{x})}}(f^*)'(w^i-c(\vb{x},\vb{z}^i)) \,\dd\H^1.
    \end{split}
\end{align}
We will handle the two integrals separately. 
First we will handle the integral with the second derivative by using the divergence theorem with
\begin{equation}
    \vb{F}=\mqty[0 \\ -z_2^i(f^*)'(w^i-c(\Phi(\vb{p}),\vb{z}^i))],
\end{equation}
as follows
\begin{align*}
    &\int_{\Lag{i}}(f^*)''(w^i-c(\vb{x},\vb{z}^i))\,\dd\H^2 \\
    &= \frac{1}{\f^2g}\int_{\Phi^{-1}(\Lag{i})}(f^*)''(w^i-c(\Phi(\vb{p}),\vb{z}^i)) \,\dd\H^2 \\
    &= \frac{1}{\f^2g}\int_{\partial\Phi^{-1}(\Lag{i})}\vb{F}\cdot\vb{n}\,\dd\H^1 \\ 
    &=\sum_{j=1}^{K^i}-\frac{z_2^in_2}{\f^2g}\norm{\vb{p}^{j+1}-\vb{p}^j}\int_0^1(f^*)'(w^i-c(\Phi(\upgamma(s)),\vb{z}^i))\,\dd s.
\end{align*}

Now we will handle the integral with the first derivative using the path parametrisation but first define
\begin{equation}
    g(\vb{x})=\norm{\nabla_{\vb{x}}u^{ik}(\vb{x})}=\sqrt{g^2\qty(\frac{1}{z^i_2}-\frac{1}{z^k_2})^2+\f^4\qty(\frac{z^i_2z^k_1-z^i_1z^k_2+x_1(z^k_2-z^i_2)}{z^i_2z^k_2})^2},
\end{equation}
letting 
\begin{align*}
    A_1 &= \frac{(p^{j+1}_1-p^j_1)^2(z^i_2-z^k_2)^2}{(z^i_2)^2(z^k_2)^2} \\
    A_2 &= \frac{2(p^{j+1}_1-p^j_1)(z^i_2-z^k_2)(p^j_1 (z^i_2 - z^k_2) + \f^2 (z^i_1 z^k_2-z^i_2 z^k_1))}{(z^i_2)^2(z^k_2)^2} \\
    A_3 &= \frac{g^2 (z^i_2 - z^k_2)^2 + (p^j_1 (z^i_2 - z^k_2) + 
   \f^2 (z^i_1 z^k_2-z^i_2 z^k_1))^2}{(z^i_2)^2 (z^k_2)^2},
\end{align*}
so that 
\begin{equation}
    g(\Phi(\upgamma(s)))=\sqrt{A_1 s^2 + A_2 s + A_3}.
\end{equation}
and therefore,
\begin{align*}
    &\int_{\Lag{i}\cap \Lag{k}}\frac{1}{\norm{\nabla_{\vb{x}}u^{ik}(\vb{x})}}(f^*)'(w^i-c(\vb{x},\vb{z}^i))\,\dd\H^1 \\
    &= \int_0^1 \frac{(f^*)'(w^i-c(\Phi(\upgamma(s)),\vb{z}^i))}{g(\Phi(\upgamma(s)))}\norm{D\Phi(\upgamma(s))\cdot\qty(\vb{p}^{j+1}-\vb{p}^j)}\,\dd s.
\end{align*}

\subsection{Centroid}\label{app:centroid}

In the two dimensional problem the centroid is given by 
\begin{equation}
    \vb{C}^i(\vb{z})=\int_{\Lag{i}}\vb{x}(f^*)'(w^i-c(\vb{x},\vb{z}^i))\,\dd\H^2.
\end{equation}

We can reduce the integral to an evaluation at the vertices as by using the divergence theorem with 
\begin{align}
    \vb{F}_1&=\mqty[0 \\ -z_2^ip_1f^*(w^i-c(\Phi(\vb{p}),\vb{z}^i))] \\
    \vb{F}_2&=\mqty[\frac{p_2z^i_2}{z^i_1}f^*(w^i-c(\Phi(\vb{p}),\vb{z}^i)) \\ 0] \\
    \vb{F}_3&=\mqty[0 \\ -z_2^ip_1^2f^*(w^i-c(\Phi(\vb{p}),\vb{z}^i))] .
\end{align}
We handle the components separately,
\begin{align*}
    &\int_{\Lag{i}}x_1(f^*)'(w^i-c(\vb{x},\vb{z}^i))\,\dd\H^2 \\ 
    &= \frac{1}{\f^2g}\int_{\Phi^{-1}(\Lag{i})}\frac{p_1}{\f^2}(f^*)'(w^i-c(\Phi(\vb{p}),\vb{z}^i)) \,\dd\H^2 \\
    &= \frac{1}{\f^4g}\int_{\partial\Phi^{-1}(\Lag{i})}\vb{F}_1\cdot\vb{n}\,\dd\H^1 \\ 
    &=\sum_{j=1}^{K^i}-\frac{z_2^in_2}{\f^4g}\norm{\vb{p}^{j+1}-\vb{p}^j}\int_0^1\upgamma_1(s)f^*(w^i-c(\Phi(\upgamma(s)),\vb{z}^i))\,\dd s.
\end{align*}

Now the other component,
\begin{align*}
    &\int_{\Lag{i}}x_2(f^*)'(w^i-c(\vb{x},\vb{z}^i))\,\dd\H^2 \\
    &= \frac{1}{\f^2g}\int_{\Phi^{-1}(\Lag{i})}\frac{1}{g}\qty(p_2-\frac{p_1^2}{2\f^2})(f^*)'(w^i-c(\Phi(\vb{p}),\vb{z}^i)) \,\dd\H^2 \\
    &= \frac{1}{\f^2g^2}\int_{\partial\Phi^{-1}(\Lag{i})}\vb{F}_2\cdot\vb{n}\,\dd\H^1-\frac{1}{2\f^4g^2}\int_{\partial\Phi^{-1}(\Lag{i})}\vb{F}_3\cdot\vb{n}\,\dd\H^1 \\ 
    &=\sum_{j=1}^{K^i}\frac{z^i_2n_1}{z^i_1\f^2g^2} \norm{\vb{p}^{j+1}-\vb{p}^j}\int_0^1\upgamma_2(s)f^*(w^i-c(\Phi(\upgamma(s)),\vb{z}^i))\,\dd s \\
    &\quad+\frac{z_2^in_2}{2\f^4g^2}\norm{\vb{p}^{j+1}-\vb{p}^j}\int_0^1(\upgamma_1(s))^2f^*(w^i-c(\Phi(\upgamma(s)),\vb{z}^i))\,\dd s.
\end{align*}

\subsection{Internal Energy}\label{app:EI}

The final integral that we need to evaluate in two dimensions is the internal energy due to the compressibility and thermodynamics. This is given by
\begin{equation}
    \Pi_0\int_{\Lag{i}}(f^*)'(w^i-c(\vb{x},\vb{z}^i))\,\dd\H^2-\qty(\frac{R_d}{p_0})^{\gamma-1}\int_{\Lag{i}}((f^*)'(w^i-c(\vb{x},\vb{z}^i)))^\gamma\,\dd\H^2.
\end{equation}
The second term is just a scaled version of the integral that we evaluated for the gradient so that one is already handled. We need to treat first term separately. By choosing
\begin{equation}
    \vb{F}=\mqty[0 \\ -\frac{z^i_2(\gamma-1)}{\qty(2\gamma-1)} (w^i-c(\Phi(\vb{p}),\vb{z}^i))((f^*)'(w^i-c(\Phi(\vb{p}),\vb{z}^i)))^{\gamma}],
\end{equation}
we can apply the divergence theorem as follows
\begin{align*}
    &\int_{\Lag{i}}((f^*)'(w^i-c(\vb{x},\vb{z}^i)))^\gamma\,\dd\H^2 \\
    &=\frac{1}{\f^2g}\int_{\Phi^{-1}(\Lag{i})}((f^*)'(w^i-c(\Phi(\vb{p}),\vb{z}^i)))^{\gamma} \,\dd\H^2 \\
    &=\frac{1}{\f^2g}\int_{\partial\Phi^{-1}(\Lag{i})}\vb{F}\cdot\vb{n} \,\dd\H^1 \\
    &=-\frac{z^i_2n_2(\gamma-1)}{\f^2g(2\gamma-1)}\int_{\partial\Phi^{-1}(\Lag{i})} (w^i-c(\Phi(\vb{p}),\vb{z}^i))\qty((f^*)'(w^i-c(\Phi(\vb{p}),\vb{z}^i)))^{\gamma} \,\dd\H^1 \\
    &=\sum_{j=1}^{K^i}-\frac{z^i_2n_2(\gamma-1)}{\f^2g(2\gamma-1)}\norm{\vb{p}^{j+1}-\vb{p}^j}\int_0^1 (w^i-c(\Phi(\upgamma(s)),\vb{z}^i))\qty((f^*)'(w^i-c(\Phi(\upgamma(s)),\vb{z}^i)))^{\gamma} \,\dd s.
\end{align*}

\end{document}